\documentstyle[amscd,amssymb,xypic,11pt]{amsart}
\xyoption{all}

\newcommand{\nc}{\newcommand}

\nc{\exto}[1]{\stackrel{#1}{\longrightarrow}}
\nc{\dlim}{{\mathop{\lim\limits_{\longrightarrow}\,}}}
\nc{\ilim}{{\mathop{\lim\limits_{\longleftarrow}\,}}}
\nc{\hocolim}{{\mathop{\sf hocolim}\,}}
\nc{\holim}{{\mathop{\sf holim}}}
\nc{\lan}{\big\langle}
\nc{\ran}{\big\rangle}

\nc{\kk}{{\mathsf{k}}}

\nc{\C}{{\mathbb{C}}}
\nc{\HH}{{\mathbb{H}}}
\nc{\PP}{{\mathbb{P}}}
\nc{\QQ}{{\mathbb{Q}}}
\nc{\ZZ}{{\mathbb{Z}}}

\nc{\CA}{{\mathcal{A}}}
\nc{\CB}{{\mathcal{B}}}
\nc{\CC}{{\mathcal{C}}}
\nc{\D}{{\mathcal{D}}}
\nc{\CE}{{\mathcal{E}}}
\nc{\CF}{{\mathcal{F}}}
\nc{\CG}{{\mathcal{G}}}
\nc{\CH}{{\mathcal{H}}}
\nc{\CL}{{\mathcal{L}}}
\nc{\CM}{{\mathcal{M}}}
\nc{\CN}{{\mathcal{N}}}
\nc{\CO}{{\mathcal{O}}}
\nc{\CP}{{\mathcal{P}}}
\nc{\CQ}{{\mathcal{Q}}}
\nc{\CR}{{\mathcal{R}}}
\nc{\CS}{{\mathcal{S}}}
\nc{\CT}{{\mathcal{T}}}
\nc{\CU}{{\mathcal{U}}}
\nc{\CV}{{\mathcal{V}}}
\nc{\CW}{{\mathcal{W}}}
\nc{\CX}{{\mathcal{X}}}
\nc{\CY}{{\mathcal{Y}}}
\nc{\CMo}{{\mathcal{M}^\circ}}
\nc{\Co}{{{C}^\circ}}

\nc{\BY}{{\overline{Y}}}
\nc{\BYD}{{\overline{Y}{}^{|D|}}}
\nc{\OZ}{{\overline{Z}}}
\nc{\bg}{{\bar{g}}}

\nc{\bq}{{\mathbf{q}}}
\nc{\BD}{{\mathbf{D}}}
\nc{\BG}{{\mathbf{G}}}
\nc{\BM}{{\mathbf{M}}}
\nc{\BP}{{\mathbf{P}}}
\nc{\BZ}{{\mathbf{Z}}}
\nc{\BPr}{{\mathsf{P}}}
\nc{\BR}{{\mathbf{R}}}
\nc{\BRO}[1]{{{\mathbf{R}}^{\circ}_{#1}}}
\nc{\BRD}[1]{{{\mathbf{R}}^{|D|}_{#1}}}
\nc{\BRP}[1]{{{\mathbf{R}}^{1}_{#1}}}
\nc{\BRTP}[1]{{{\mathbf{\tilde{R}}}{}^{1}_{#1}}}
\nc{\BS}{{\mathbf{S}}}
\nc{\BMS}{{{\mathbf{M}}^{{s}}}}
\nc{\BMSS}{{{\mathbf{M}}^{{ss}}}}
\nc{\BMZ}{{\mathbf{M}^{\circ}}}
\nc{\BCL}{{\mathbf{L}}}

\nc{\PCC}{{{}^\perp\CC}}

\nc{\Cl}{{\mathsf{Cliff}}}
\nc{\Clev}{{\mathop{\mathsf{Cliff}}^{\circ}}}

\nc{\FA}{{\mathfrak{A}}}
\nc{\FB}{{\mathfrak{B}}}

\nc{\fa}{{\mathfrak{a}}}
\nc{\fb}{{\mathfrak{b}}}
\nc{\fg}{{\mathfrak{g}}}
\nc{\fn}{{\mathfrak{n}}}
\nc{\fp}{{\mathfrak{p}}}
\nc{\FD}{{\mathfrak{D}}}
\nc{\FE}{{\mathfrak{E}}}
\nc{\FL}{{\mathfrak{L}}}
\nc{\FM}{{\mathfrak{M}}}
\nc{\FS}{{\mathsf{S}}}

\nc{\sfc}{{\mathsf{c}}}
\nc{\sfch}{{\mathsf{ch}}}
\nc{\sfh}{{\mathsf{h}}}

\nc{\SK}{{\mathsf{K}}}
\nc{\SM}{{\mathsf{M}}}
\nc{\SO}{{\mathsf{O}}}
\nc{\SQ}{{\mathsf{Q}}}
\nc{\SPV}{{\mathsf{S}^+\mathsf{V}}}
\nc{\SMV}{{\mathsf{S}^-\mathsf{V}}}
\nc{\SPMV}{{\mathsf{S}^\pm\mathsf{V}}}
\nc{\SX}{{S_X}}
\nc{\SY}{{S_Y}}
\nc{\phipsi}{{q}}
\nc{\eps}{\varepsilon}

\nc{\pim}{{\pi_-}}
\nc{\pip}{{\pi_+}}

\nc{\BE}{{\overline{\CE}}}
\nc{\TE}{{\tilde{\CE}}}
\nc{\TQ}{{\tilde{Q}}}
\nc{\TCF}{{\tilde{\CF}}}
\nc{\TCG}{{\tilde{\CG}}}
\nc{\TCL}{{\tilde{\CL}}}
\nc{\TF}{{\tilde{F}}}
\nc{\TW}{{\tilde{W}}}
\nc{\TCC}{{\tilde{\CC}}}
\nc{\TCX}{{\tilde{\CX}}}
\nc{\TCY}{{\tilde{\CY}}}
\nc{\TPhi}{{\tilde{\Phi}}}
\nc{\OPhi}{{\bar{\Phi}}}
\nc{\txi}{{\tilde{\xi}}}
\nc{\tp}{{\tilde{p}}}
\nc{\tq}{{\tilde{q}}}
\nc{\tzeta}{{\tilde{\zeta}}}
\nc{\tpi}{{\tilde{\pi}}}

\nc{\halpha}{{\hat{\alpha}}}
\nc{\HCA}{{\hat{\CA}}}
\nc{\HCB}{{\hat{\CB}}}
\nc{\HCC}{{\hat{\CC}}}
\nc{\HE}{{\widehat{\CE}}}
\nc{\HX}{{\hat{X}}}
\nc{\hxi}{{\hat{\xi}}}

\nc{\UH}{{\mathcal{H}}}

\nc{\TM}{{\widetilde{M}}}
\nc{\TCM}{{\widetilde{\CM}}}
\nc{\TU}{{\widetilde{U}}}
\nc{\TX}{{\widetilde{X}}}
\nc{\TY}{{\widetilde{Y}}}
\nc{\TYO}{{{\widetilde{Y}}^\circ}}
\nc{\barf}{{\bar{f}}}
\nc{\te}{{\tilde{e}}{}}
\nc{\tf}{{\tilde{f}}}
\nc{\tg}{{\tilde{g}}}
\nc{\ti}{{\tilde{\imath}}}
\nc{\tj}{{\tilde{\jmath}}}
\nc{\ty}{{\tilde{y}}}
\nc{\tphi}{{\tilde{\phi}}}

\nc{\urho}{{\underline{\rho}}}

\nc{\LRA}{\Leftrightarrow}
\nc{\RA}{\Rightarrow}
\nc{\lotimes}{\mathbin{\mathop{\otimes}\limits^{\mathbb{L}}}}
\nc{\CEnd}{\mathop{\mathcal{E}\mathit{nd}}\nolimits}
\nc{\CExt}{\mathop{\mathcal{E}\mathit{xt}}\nolimits}
\nc{\CHom}{\mathop{\mathcal{H}\mathit{om}}\nolimits}
\nc{\RH}{\mathop{{\mathsf{R}}\Gamma}\nolimits}
\nc{\RGamma}{\mathop{{\mathsf{R}}\Gamma}\nolimits}
\nc{\RHom}{\mathop{\mathsf{RHom}}\nolimits}
\nc{\RCHom}{\mathop{\mathsf{R}\mathcal{H}\mathit{om}}\nolimits}
\nc{\RG}{\mathop{\mathsf{R\Gamma}}\nolimits}
\nc{\Hom}{\mathop{\mathsf{Hom}}\nolimits}
\nc{\Ext}{\mathop{\mathsf{Ext}}\nolimits}
\nc{\End}{\mathop{\mathsf{End}}\nolimits}
\nc{\Tor}{\mathop{\mathsf{Tor}}\nolimits}
\nc{\Tordim}{\mathop{\mathsf{Tor}\text{\rm-}\mathsf{dim}}\nolimits}
\nc{\Hilb}{\mathop{\mathsf{Hilb}}\nolimits}
\nc{\Spec}{\mathop{\mathsf{Spec}}\nolimits}
\nc{\Pic}{\mathop{\mathsf{Pic}}\nolimits}
\renewcommand{\Im}{\mathop{\mathsf{Im}}\nolimits}
\nc{\Tw}{\mathop{\mathsf{Tw}}\nolimits}
\nc{\Cone}{\mathop{\mathsf{Cone}}\nolimits}
\nc{\Fiber}{\mathop{\mathsf{Fiber}}\nolimits}
\nc{\Ker}{\mathop{\mathsf{Ker}}\nolimits}
\nc{\Coker}{\mathop{\mathsf{Coker}}\nolimits}
\nc{\codim}{\mathop{\mathsf{codim}}\nolimits}
\nc{\sing}{{\mathsf{sing}}}
\nc{\supp}{\mathop{\mathsf{supp}}}
\nc{\perf}{{\mathsf{perf}}}
\nc{\qperf}{{\mathsf{q\text{-}perf}}}
\nc{\QPerf}{{\mathsf{Q\text{-}Perf}}}
\nc{\rank}{\mathop{\mathsf{rank}}}
\nc{\Pf}{{\mathsf{Pf}}}
\nc{\Gr}{{\mathsf{Gr}}}
\nc{\OGr}{{\mathsf{OGr}}}
\nc{\Flag}{{\mathsf{Fl}}}
\nc{\Kosz}{{\mathsf{Kosz}}}
\nc{\LGr}{{\mathsf{LGr}}}
\nc{\GTGr}{{\mathsf{G_2Gr}}}
\nc{\GTF}{{\mathsf{G_2F}}}
\nc{\OF}{{\mathsf{OF}}}
\nc{\Fl}{{\mathsf{Fl}}}
\nc{\Bl}{{\mathsf{Bl}}}
\nc{\GL}{{\mathsf{GL}}}
\nc{\PGL}{{\mathsf{PGL}}}
\nc{\SL}{{\mathsf{SL}}}
\nc{\SP}{{\mathsf{Sp}}}
\nc{\Spin}{{\mathsf{Spin}}}
\nc{\Tot}{{\mathsf{Tot}}}
\nc{\ev}{{\mathsf{ev}}}
\nc{\od}{{\mathsf{odd}}}
\nc{\coev}{{\mathsf{coev}}}
\nc{\id}{{\mathsf{id}}}
\nc{\opp}{{\mathsf{opp}}}
\nc{\PS}{{{\PP^3}}}
\nc{\Qu}{{{Q^3}}}
\nc{\tdim}{\mathop{\Tor\dim}}
\nc{\ecart}{{\fbox{$\scriptstyle\mathsf{EC}$}}}
\nc{\ad}{{\mathop{\mathsf ad}}}
\nc{\sg}{{\mathop{\mathsf sg}}}
\nc{\hf}{{\mathop{\mathsf hf}}}
\nc{\gr}{{\mathop{\mathsf gr}}}
\nc{\qgr}{{\mathop{\mathsf qgr}}}
\nc{\Coh}{{\mathop{\mathsf Coh}}}
\nc{\Ab}{{\mathop{\mathcal{A}\mathit{b}}}}
\nc{\Ccoh}{{\mathop{\mathsf Ccoh}}}
\nc{\Qcoh}{{\mathop{\mathsf Qcoh}}}

\nc{\AAV}{{\mathcal{AAV}}}

\nc{\Rep}{{\mathsf{Rep}}}

\nc{\Cubics}{{{\mathcal{S}}_3}}
\nc{\VFT}{{{\mathcal{S}}_{14}}}
\nc{\VFTE}{{{\mathcal{N}}_{\mathrm{reg,sm}}}}
\nc{\MX}{{\CM_X}}
\nc{\MY}{{\CM_Y}}
\nc{\MYE}{{\CM_{Y,\CE}}}
\nc{\Yd}{{Y_d}}
\nc{\Yfive}{{Y_5}}
\nc{\Xg}{{X_{2g-2}}}
\nc{\Xtt}{{X_{22}}}
\nc{\Xst}{{X_{16}}}
\nc{\Xtw}{{X_{12}}}
\nc{\Xe}{{X_{8}}}
\nc{\Xf}{{X_{4}}}

\nc{\git}{{/\!\!/\!{}_\chi}}

\theoremstyle{plain}

\newtheorem{theorem}{Theorem}[section]

\newtheorem{lemma}[theorem]{Lemma}
\newtheorem{proposition}[theorem]{Proposition}
\newtheorem{corollary}[theorem]{Corollary}

\theoremstyle{definition}

\newtheorem{definition}[theorem]{Definition}

\theoremstyle{remark}

\newtheorem{remark}[theorem]{Remark}

\newenvironment{proof}{\noindent{\sf Proof:}}{\qed\medskip}

\title{Base change for semiorthogonal decompositions}
\author{Alexander Kuznetsov}
\address{\sloppy
\parbox{0.9\textwidth}{
Algebra Section, Steklov Mathematical Institute, 8 Gubkin str., Moscow 119991 Russia\hfill\\[5pt]
The Poncelet Laboratory, Independent University of Moscow\hfill\\[5pt]
}}
\email{akuznet@@mi.ras.ru}
\date{}
\thanks{I was partially supported by
RFFI grants 07-01-00051, 07-01-92211, and 08-01-00297,
NSh 1987.2008.1,
Russian Presidential grant for young scientists No. MD-2712.2009.1.}

\begin{document}

\begin{abstract}
Let $X$ be an algebraic variety over a base scheme $S$ and $\phi:T \to S$ a base change.
Given an admissible subcategory $\CA$ in $\D^b(X)$, the bounded derived category of coherent
sheaves on $X$, we construct under some technical conditions an admissible subcategory $\CA_T$
in $\D^b(X\times_S T)$, called the base change of $\CA$, in such a way that the following base
change theorem holds: if a semiorthogonal decomposition of $\D^b(X)$ is given then the base changes
of its components form a semiorthogonal decomposition of $\D^b(X\times_S T)$. As an intermediate step
we construct a compatible system of semiorthogonal decompositions of the unbounded derived category of
quasicoherent sheaves on $X$ and of the category of perfect
complexes on $X$. As an application we prove that the projection
functors of a semiorthogonal decomposition are kernel functors.
\end{abstract}

\maketitle

\section{Introduction}

An important approach to the noncommutative algebraic geometry is to consider
triangulated categories with good properties as substitutes for noncommutative
varieties. Given such category we consider it as the bounded derived category
of coherent sheaves on a would-be variety and try to do some geometry. Note however that
even the simplest geometric functors between derived categories often do not preserve
boundedness or coherence --- the pullback functor preserves boundedness only if the corresponding
morphism has finite $\Tor$-dimension and the pushforward functor
preserves coherence only if the corresponding map is proper.
So, to do noncommutative geometry we need some unbounded and quasicoherent versions
of triangulated categories under consideration. One goal of this paper is the following:
given a good triangulated category $\CA$ (considered as a bounded derived category
of coherent sheaves) to define a category $\CA_{qc}$ which is a substitute for the unbounded
derived category of quasicoherent sheaves and a category $\CA^-$, a substitute for the bounded
above derived category of coherent sheaves.

A straightforward approach to construct $\CA_{qc}$ would be just to consider the closure of $\CA$ under colimits.
However it is not clear how to define a triangulated structure there. So, instead, we assume
that the category $\CA$ is given as an admissible subcategory in $\D^b(X)$, the bounded derived
category of coherent sheaves on some algebraic variety $X$, and consider the minimal triangulated
subcategory $\HCA \subset \D_{qc}(X)$ containing $\CA$ and closed under arbitrary direct sums.
Defined this way the category $\HCA$
inherits a triangulated structure automatically, but there arises a question of dependence of $\HCA$ on the choice
of the variety $X$ and of the embedding $\CA \to \D^b(X)$. We prove that it is actually independent
of these choices under some technical condition.

Another, and in fact the most important goal of the paper, is to define a base change
for triangulated categories. Assume that $S$ is an algebraic variety and $\CA$ is
a good triangulated category over $S$ (which can be understood, for example,
as that $\CA$ is a module category over the tensor triangulated category $\D^\perf(S)$
of perfect complexes on $S$). Given a base change $\phi:T \to S$ we would like to define
a triangulated category $\CA_T$ over $T$ to be considered as the base change of $\CA$.
Again, an abstract approach is too complicated, so we assume that $\CA$ is given as
an $S$-linear admissible subcategory in $\D^b(X)$ ($S$-linear means closed under tensoring
with pullbacks of perfect complexes on $S$), where $X$ is an algebraic variety over $S$,
and construct $\CA_T$ as a certain triangulated subcategory in $\D^b(X\times_S T)$.
Once again there arises an issue of dependance on the chosen embedding $\CA \to \D^b(X)$,
and again we show that the result is independent of the choice.

The most important technical notion used in the paper is that of a semiorthogonal decomposition.
Actually, we start not with an admissible subcategory $\CA \subset \D^b(X)$ but with a semiorthogonal
decomposition $\D^b(X) = \lan \CA_1,\CA_2,\dots,\CA_m \ran$. Then we consider a chain of
triangulated categories $\D^\perf(X) \subset \D^b(X) \subset \D^-(X) \subset \D_{qc}(X)$
(here $\D^-(X)$ is the derived category of bounded above complexes with coherent cohomology)
and ask whether there exist semiorthogonal decompositions of these categories compatible
with the initial decomposition. It turns out that the categories $\CA_i^\perf = \CA_i \cap \D^\perf(X)$
always give a semiorthogonal decomposition of $\D^\perf(X)$, while the categories
$\HCA_i$ (the minimal triangulated subcategories of $\D_{qc}(X)$ containing $\CA_i^\perf$
and closed under arbitrary direct sums) and $\CA_i^- = \HCA_i \cap \D^-(X)$
always form semiorthogonal decompositions of $\D_{qc}(X)$ and $\D^-(X)$ respectively.
However, for compatibility of the last two decompositions with the initial decomposition of $\D^b(X)$,
we need a technical condition to be satisfied, namely the right cohomological
amplitude of the projection functors of the initial decomposition should be finite
(this condition holds automatically if $X$ is smooth).

Similarly, in a situation of a base change we start with a semiorthogonal decomposition
of $\D^b(X)$. However, here we need some additional assumptions from the very beginning.
First of all the decomposition of $\D^b(X)$ should be $S$-linear, and second,
the base change $\phi:T \to S$ should be faithful for the projection $f:X \to S$.
The latter condition more or less by definition (see~\cite{K1}) is equivalent to the base change isomorphism
$f_*\phi^* \cong \phi^*f_*$, where the projections of $X_T = X\times_S T$ to $X$ and $T$
by an abuse of notation are denoted by $\phi$ and $f$ respectively.

The semiorthogonal decomposition of $\D^b(X_T)$ is constructed in several steps. First, we consider
the semiorthogonal decomposition of $\D^\perf(X)$ constructed above. Then we define
the subcategory $\CA_{iT}^p$ of $\D^\perf(X_T)$ to be the closed under direct summands
triangulated subcategory generated by objects of the form $\phi^*F \otimes f^*G$ with $F \in \CA_i^\perf$
and $G \in \D^\perf(T)$. It turns out that acting this way we always obtain a semiorthogonal
decomposition of $\D^\perf(X_T)$. Further, we define the category $\HCA_{iT}$
to be the minimal triangulated subcategory of $\D_{qc}(X_T)$ containing $\CA_{iT}^p$ and closed
under arbitrary direct sums, and $\CA_{iT}^- = \HCA_{iT} \cap \D^-(X_T)$. Thus we obtain
semiorthogonal decompositions of $\D_{qc}(X_T)$ and $\D^-(X_T)$. Finally we consider subcategories
$\CA_{iT} = \CA_{iT}^- \cap \D^b(X_T) \subset \D^b(X_T)$. But to prove that they form
a semiorthogonal decomposition we again need the assumption of finiteness of
cohomological amplitude of the projection functors of the initial semiorthogonal
decomposition of $\D^b(X)$. We prove that the projection functors of the obtained
decomposition of $\D^b(X_T)$ also have finite cohomological amplitude.

We show that the constructed semiorthogonal decompositions of $\D_{qc}(X)$ and $\D_{qc}(X_T)$
are compatible with respect to the pushforward and the pullback functors via the projection
$\phi:X_T \to X$. It follows, that the semiorthogonal decompositions of $\D^b(X)$ and $\D^b(X_T)$
are compatible with respect to $\phi_*$ whenever $\phi$ is proper, and with respect to $\phi^*$
whenever $\phi$ has finite $\Tor$-dimension.

It should be mentioned, that seemingly too complicated procedure of constructing $\CA_{iT}$
is probably inevitable. The straightforward approach of taking for $\CA_{iT}$ the subcategory
of $\D^b(X_T)$ generated by objects of the form $\phi^*F \otimes f^*G$ with $F \in \CA_i$
and $G \in \D^b(T)$ doesn't give the desired result even when both $\phi$ and $f$ have finite $\Tor$-dimension.
Indeed, assume that $\CA_i = \D^b(X)$ and $X$ is smooth. Then $\D^b(X) = \D^\perf(X)$ and it is clear
that defined this way subcategory of $\D^b(X_T)$ is just the category of perfect complexes
$\D^\perf(X_T)$, not the whole $\D^b(X)$ as one would wish. So, one definitely needs to add
something to this category to obtain the right answer. It seems that to add all colimits
and then to intersect with $\D^b(X_T)$ is the simplest possible solution. And considering perfect
complexes as an intermediate step both removes many technical problems and gives an additional information.

As an application of the obtained results we prove the following. Assume that
$\D^b(X) = \lan \CA_1,\dots,\CA_m \ran$ is a semiorthogonal decomposition
the projection functors of which have finite cohomological amplitude.
We prove then that these functors are isomorphic to kernel functors $\Phi_{K_i}$
given by some explicit kernels $K_i \in \D^b(X\times X)$. In particular,
if $\CA \subset \D^b(X)$ is an admissible subcategory and the projection
functor to $\CA$ has finite cohomological amplitude then it is isomorphic
to a kernel functor. In a special case, when $\CA \cong \D^b(Y)$ for a smooth
projective variety $Y$ this follows from the Orlov's Theorem on representability
of fully faithful functors~\cite{O1}. Indeed, in this case the embedding functor
$\D^b(Y) \to \D^b(X)$ as well as its adjoint are given by appropriate kernels on $X\times Y$,
so the projection functor is given by the convolution of these kernels. Thus, our
result can be considered as a generalization of Orlov's Theorem.

The paper is organized as follows.
In Section~2 we remind the main technical notions used in the paper ---
semiorthogonal decompositions, cohomological amplitude, homotopy colimits e.t.c.
We also discuss several notions and facts related to approximation of unbounded
quasicoherent complexes by perfect ones.
%
In Section~3 we investigate when a semiorthogonal decomposition of a triangulated
category $\CT'$ induces a semiorthogonal decomposition of its full triangulated
subcategory $\CT \subset \CT'$.
In Section~4 we construct extensions of a semiorthogonal decomposition of $\D^b(X)$
to $\D^\perf(X) \subset \D^-(X) \subset \D_{qc}(X)$.
In Section~5 we define the base change for an admissible subcategory
and prove the faithful base change Theorem.
In Section~6 we show that extensions $\HCA$, $\CA^-$ and the base change $\CA_T$ of $\CA$
do not depend on the choice of $X$ and of the embedding $\CA \to \D^b(X)$ involved in the definitions.
In Section~7 we prove that the projection functors of a semiorthogonal decomposition
can be represented as kernel functors.

{\bf Acknowledgements:}
I would like to thank A.~Bondal, D.~Kaledin, D.~Orlov and L.~Positselski for very helpful discussions.
My special thanks to A.~Elagin who suggested a significant improvement of the previous version of this paper.

\section{Preliminaries}

\subsection{Notation}

All algebraic varieties are assumed to be quasiprojective.

For an algebraic variety $X$, we denote
by $\D^b(X)$ the bounded derived category of coherent sheaves on $X$,
by $\D^-(X)$ the bounded above derived category of coherent sheaves on $X$, and
by $\D_{qc}(X)$ the unbounded derived category of quasicoherent sheaves on $X$.
Recall that an object $F \in \D_{qc}(X)$ is a {\sf perfect complex}\/ if it is locally
quasiisomorphic to a bounded complex of locally free sheaves of finite rank.
Recall that perfect complexes are precisely {\sf compact objects} in $\D_{qc}(X)$, i.e.
if $P$ is perfect then
$$
\Hom(P,\oplus_\alpha F_\alpha) \cong \oplus_\alpha \Hom(P,F_\alpha)
$$
for any system $F_\alpha \in \D_{qc}(X)$.
We denote by $\D^\perf(X)$ the full subcategory of $\D_{qc}(X)$ consisting
of perfect complexes. Note that $\D^\perf(X)$ is a triangulated subcategory in $\D^b(X)$.
Given an object $F \in \D_{qc}(X)$ we denote by $\CH^i(F)$ the $i$-th cohomology sheaf of~$F$.

For $F,G \in \D_{qc}(X)$, we denote by $\RCHom(F,G)$ the local $\RCHom$-complex
and by $F \otimes G$ the derived tensor product. Similarly, for a map $f:X \to Y$,
we denote by $f_*:\D_{qc}(X) \to \D_{qc}(Y)$ the derived pushforward functor and by
$f^*:\D_{qc}(Y) \to \D_{qc}(X)$ the derived pullback functor. We refer to~\cite{KSch}
for the definition of these functors. We also denote by $f^!:\D_{qc}(Y) \to \D_{qc}(X)$
the right adjoint functor of $f_*$ (usually it is referred to as the twisted pullback functor).
It exists by~\cite{N2} (see also~\cite{KSch}). If the morphism $f$ is smooth then
$f^!(F) \cong f^*(F)\otimes\omega_{X/Y}[\dim X - \dim Y]$ again by~\cite{N2}.

Given a class $\CE$ of objects in a triangulated category $\CT$
we denote by $\lan\CE\ran$ the minimal strictly full triangulated
subcategory in $\CT$ containing all objects in $\CE$ and closed
under taking direct summands. We say that $\CE$ {\sf generates}\/ $\CT$
if $\CT = \lan\CE\ran$.

\subsection{Semiorthogonal decompositions}

Given a class $\CE$ of objects in a triangulated category $\CT$ we denote
the {\sf right}\/ and the {\sf left orthogonal}\/ to $\CE$ by
$$
\begin{array}{rcl}
\CE^\perp   &=& \{T \in \CT\ |\ \Hom(E[k],T) = 0\ \text{for all $E\in\CE$ and all $k \in \ZZ$}\},\\
{}^\perp\CE &=& \{T \in \CT\ |\ \Hom(T,E[k]) = 0\ \text{for all $E\in\CE$ and all $k \in \ZZ$}\}.
\end{array}
$$
It is clear that both $\CE^\perp$ and ${}^\perp\CE$ are triangulated subcategories in $\CT$ closed
under taking direct summands.
The classes $\CE_1,\CE_2 \subset \CT$ are called {\sf semiorthogonal}\/
if $\CE_1 \subset \CE_2^\perp$, or equivalently $\CE_2 \subset {}^\perp\CE_1$.

\begin{lemma}\label{lrso}
If classes $\CE_1$ and $\CE_2$ are semiorthogonal
then the subcategories $\lan\CE_1\ran$ and $\lan\CE_2\ran$ are semiorthogonal as well.
\end{lemma}
\begin{proof}
We have $\CE_1 \subset \CE_2^\perp$, hence $\lan\CE_1\ran \subset \CE_2^\perp$,
hence $\CE_2 \subset {}^\perp\lan\CE_1\ran$, hence $\lan\CE_2\ran \subset {}^\perp\lan\CE_1\ran$.
\end{proof}

\begin{definition}[\cite{BK,BO1,BO2}]
A {\sf semiorthogonal decomposition}\/ of a triangulated category $\CT$ is a sequence
of full triangulated subcategories
$\CA_1,\dots,\CA_m$ in $\CT$ such that $\CA_i \subset \CA_j^\perp$ for $i < j$
and for every object $T \in \CT$ there exists a chain of morphisms
$0 = T_m \to T_{m-1} \to \dots \to T_1 \to T_0 = T$ such that
the cone of the morphism $T_k \to T_{k-1}$ is contained in $\CA_k$
for each $k=1,2,\dots,m$. In other words, there exists a diagram
\begin{equation}\label{tower}
\vcenter{
\xymatrix@C-7pt{
0 \ar@{=}[r] & T_m \ar[rr]&& T_{m-1} \ar[dl] \ar[rr]&& \quad\dots\quad \ar[rr]&& T_2 \ar[rr]&& T_1 \ar[dl] \ar[rr]&& T_0 \ar[dl] \ar@{=}[r] & T \\
&& A_m \ar@{..>}[ul] &&& \dots &&& A_2 \ar@{..>}[ul]&& A_1 \ar@{..>}[ul]&&
}}
\end{equation}
where all triangles are distinguished (dashed arrows have degree $1$) and $A_k \in \CA_k$.
\end{definition}

Thus, every object $T\in\CT$ admits a decreasing ``filtration''
with factors in $\CA_1$, \dots, $\CA_m$ respectively.

\begin{lemma}\label{ff}
If $\CT = \lan \CA_1, \dots, \CA_m \ran$ is a semiorthogonal decomposition and $T \in \CT$
then the diagram~\eqref{tower} for $T$ is unique and functorial (for any morphism $T \to T'$
there exists a unique collection of morphisms $T_i \to T'_i$, $A_i \to A'_i$ combining into
a morphism of diagram~\eqref{tower} for $T$ into diagram~\eqref{tower} for $T'$).
\end{lemma}
\begin{proof}
Note that $T_1 \in \lan \CA_2, \dots, \CA_m \ran$ by~\eqref{tower}.
It follows from the semiorthogonality that $\Hom(T_1,A'_1[k]) = 0$ for all $k \in \ZZ$.
Therefore any map $T_0 = T \to T' = T'_0$ extends in a unique way to a map of the triangle
$T_1 \to T_0 \to A_1$ into the triangle $T'_1 \to T'_0 \to A'_1$. In particular, we obtain
a map $T_1 \to T'_1$ as well as a map $A_1 \to A'_1$ and proceed by induction.
\end{proof}

We denote by $\alpha_k:\CT \to \CT$ the functor $T \mapsto A_k$.
We call $\alpha_k$ the {\sf $k$-th projection functor}\/ of the semiorthogonal decomposition.

\begin{definition}[\cite{BK,B}]
A full triangulated subcategory $\CA$ of a triangulated category $\CT$ is called
{\sf right admissible}\/ if for the inclusion functor $i:\CA \to \CT$ there is
a right adjoint $i^!:\CT \to \CA$, and
{\sf left admissible}\/ if there is a left adjoint $i^*:\CT \to \CA$.
Subcategory $\CA$ is called {\sf admissible}\/ if it is both right and left admissible.
\end{definition}

\begin{lemma}[\cite{B}]\label{sod_adm}
If $\CT = \lan\CA,\CB\ran$ is a semiorthogonal decomposition then
$\CA$ is left admissible and $\CB$ is right admissible.
Conversely, if $\CA \subset \CT$ is left admissible then
$\CT = \lan \CA,{}^\perp\CA \ran$ is a semiorthogonal decomposition, and
if $\CB \subset \CT$ is right admissible then
$\CT = \lan \CB^\perp,\CB \ran$ is a semiorthogonal decomposition.
\end{lemma}

\begin{definition}\label{strongsod}
We will say that a semiorthogonal decomposition $\CT = \lan \CA_1, \dots, \CA_m \ran$
is a {\sf strong semiorthogonal decomposition}\/ if for each $k$ the category
$\CA_k$ is admissible in $\lan \CA_k, \dots, \CA_m \ran$.
\end{definition}

Note that $\CA_k$ is left admissible in $\lan \CA_k, \dots, \CA_m \ran$ by Lemma~\ref{sod_adm}.
So the additional condition in the definition is the right admissibility.
Note also that if $\CA_k$ is right admissible in $\CT$ then it is also admissible in $\lan \CA_k, \dots, \CA_m \ran$
(thus a semiorthogonal decomposition with admissible components is a strong semiorthogonal decomposition),
and that in the case when $\CT = \D^b(X)$ with $X$ being smooth and projective any semiorthogonal
decomposition is strong.

\subsection{$S$-linearity}

Let $f:X \to S$ be a morphism of algebraic varieties.
A triangulated subcategory $\CA \subset \D_{qc}(X)$ is called {\sf $S$-linear}\/ (see \cite{K1})
if it is stable with respect to tensoring by pullbacks of perfect complexes on $S$.
In other words, if $A \otimes f^* F \in \CA$ for any $A \in \CA$, $F \in \D^\perf(S)$.

\begin{lemma}\label{slinso}
A pair of $S$-linear subcategories $\CA,\CB \subset \D_{qc}(X)$ is semiorthogonal if and only if
the equality $f_*\RCHom(B,A) = 0$ holds for any $A \in \CA$, $B \in \CB$.
\end{lemma}
\begin{proof}
First we note that for any object $0 \ne G \in \D_{qc}(S)$ there exists a nonzero map $P \to G$
from a perfect complex $P \in \D^\perf(S)$. Indeed, represent $G$ by a complex of quasicoherent
sheaves and assume that $\CH^i(G) \ne 0$. Let $Z^i = \Ker(G^i \to G^{i+1})$ so that we have an
epimorphism $Z^i \to \CH^i(G)$. It is clear that there exists a locally free sheaf $P$ of finite rank and a map
$P \to Z^i$ such that the composition $P \to Z^i \to \CH^i(G)$ is nonzero. Then the composition
$P \to Z^i \subset G^i$ induces the required morphism $P[-i] \to G$ (it is nonzero since
the induced morphism of the cohomology $\CH^i(P[-i]) = P \to \CH^i(G)$ is nonzero).

Further note that $\RHom(P,f_*\RCHom(B,A)) \cong \RHom(f^*P,\RCHom(B,A)) \cong \RHom(B\otimes f^*P,A)$
for any $P \in \D^\perf(S)$. So, if $\CA$ and $\CB$ are semiorthogonal then $\RHom(B\otimes f^*P,A) = 0$
since $\CB$ is $S$-linear and the above observation shows that $f_*\RCHom(B,A) = 0$. The inverse
is evident.
\end{proof}

Let $f:X \to S$ and $g:Y \to S$ be algebraic morphisms, and
assume that $\CA \subset \D_{qc}(X)$, $\CB \subset \D_{qc}(Y)$ are $S$-linear
triangulated subcategories. A functor $\Phi:\CA \to \CB$ is called {\sf $S$-linear}\/
if there is given a functorial isomorphism $\Phi(F\otimes f^*G) \cong \Phi(F) \otimes g^*G$
for all $F \in \CA$, $G \in \D^\perf(S)$.

\begin{lemma}\label{prfsl}
If\/ $\CT \subset \D_{qc}(X)$ is an $S$-linear triangulated subcategory
and $\CT = \lan \CA_1, \dots, \CA_m \ran$ is an $S$-linear semiorthogonal
decomposition then its projection functors $\alpha_i:\CT \to \CT$ are $S$-linear.
\end{lemma}
\begin{proof}
Take any $G \in \D^\perf(S)$ and consider the endofunctor of $\CT$
given by tensoring with $f^*G$. It preserves all $\CA_i$ hence by Lemma~\ref{cc} below
it commutes with the projection functors. This gives the required functorial isomorphism.
\end{proof}

\subsection{Faithful base changes}\label{ssfbc}

Let $f:X \to S$ and $\phi:T \to S$ be algebraic morphisms.
Let $X_T = X\times_T S$ be the fiber product. By an abuse of notation
denote the projections $X_T \to T$ and $X_T \to X$ also by $f$ and $\phi$ respectively.
It is easy to see that there is a canonical morphism of functors $\phi^*f_* \to f_*\phi^*$.
Recall that the cartesian square
$$
\xymatrix{
X_T \ar[r]^\phi \ar[d]_f & X \ar[d]^f \\ T \ar[r]^\phi & S
}
$$
is called {\sf exact}\/ (see~\cite{K1}) if this morphism of functors
is an isomorphism. By \cite{K1} the square is exact if either $f$ or $\phi$ is flat,
and the square is exact if and only if the transposed square is exact.

A map $\phi:T \to S$ considered as a change of base is called
{\sf faithful for $f:X \to S$} (see~\cite{K1}) if the corresponding cartesian square is exact.
Thus any change of base is faithful for a flat $f$ and similarly a flat change of base
is faithful for any $f$.

\subsection{Truncations}

Given a complex $C^\bullet$ its {\sf stupid truncations}\/ are defined as
$$
(\sigma^{\le m} C)^n =
\begin{cases}
C^n, & \text{if $n \le m$}\\
0, & \text{if $n > m$}
\end{cases}
\qquad\text{and}\qquad
(\sigma^{\ge m} C)^n =
\begin{cases}
C^n, & \text{if $n \ge m$}\\
0, & \text{if $n < m$}
\end{cases}
$$
It is clear that $\sigma^{\ge m}C \to C \to \sigma^{\le m-1}C$ is a distinguished triangle
in the derived category. The advantage of the stupid truncations which we will use
subsequently in the paper is that when applied to a complex of locally free sheaves
(a perfect complex) they produce a perfect complex as well.

Similarly, the {\sf canonical truncations}\/ (also known as {\sf smart truncations}) are defined  as
$$
(\tau^{\le m} C)^n =
\begin{cases}
C^n, & \text{if $n < m$}\\
\Ker (d: C^{m} \to C^{m+1}), & \text{if $n = m$}\\
0, & \text{if $n > m$}
\end{cases}
\qquad
(\tau^{\ge m} C)^n =
\begin{cases}
C^n, & \text{if $n > m$}\\
\Coker(d: C^{m-1} \to C^m), & \text{if $n = m$}\\
0, & \text{if $n < m$}
\end{cases}
$$
Again, in the derived category we have a distinguished triangle
$\tau^{\le m}C \to C \to \tau^{\ge m+1}C$. The advantage of the canonical truncations
is that they descend to functors on the derived category.
Note also that
$$
\CH^n(\tau^{\le m} C) \cong
\begin{cases}
\CH^n(C), & \text{if $n \le m$}\\
0, & \text{if $n > m$}
\end{cases}
\qquad\text{and}\qquad
\CH^n(\tau^{\ge m} C) \cong
\begin{cases}
\CH^n(C), & \text{if $n \ge m$}\\
0, & \text{if $n < m$}
\end{cases}
$$

\subsection{Cohomological amplitude}

Let $\D^{[p,q]}_{qc}(X)$ denote the full subcategory of $\D_{qc}(X)$
consisting of all complexes $F \in \D_{qc}(X)$ with $\CH^i(F) = 0$ for $i \not\in [p,q]$.
Let $\CT \subset \D_{qc}(X)$ be a triangulated subcategory.
We say that $(a,b)$ is the {\sf cohomological amplitude} of a triangulated functor $\Phi:\CT \to \D_{qc}(Y)$ if
$$
\Phi(\CT \cap \D^{[p,q]}_{qc}(X)) \subset \D^{[p+a,q+b]}_{qc}(Y)
$$
for all $p,q\in\ZZ$. In particular, we say that $\Phi$ has finite left (resp.\ right) cohomological amplitude
if $a > -\infty$ (resp.\ $b < \infty$). If both $a$ and $b$ are finite we say that $\Phi$ has finite cohomological
amplitude.

\begin{lemma}\label{fca}
Every exact functor $\Phi:\D^\perf(X) \to \D_{qc}(Y)$ has finite cohomological amplitude.
\end{lemma}
\begin{proof}
The same as in \cite{K3}, Proposition 2.5. A smoothness of $X$ is not required since
we consider only perfect complexes on $X$.
\end{proof}

Let $X$ and $Y$ be algebraic varieties. Consider the product $X\times Y$ and let $p:X\times Y \to X$
and $q:X\times Y \to Y$ be the projections.
Recall (see~\cite{K1}, 10.39) that an object $K \in \D^b(X\times Y)$ {\sf has finite $\Tor$-amplitude over $X$}\/
if the functor $F \mapsto K \otimes p^*F$ has finite cohomological amplitude.
Similarly, an object $K \in \D^b(X\times Y)$ {\sf has finite $\Ext$-amplitude over $Y$}\/
if the functor $F \mapsto \RCHom(K,q^!F)$ has finite cohomological amplitude.

\begin{lemma}\label{ftefca}
If $K \in \D^b(X\times Y)$ has finite $\Tor$-amplitude over $X$ then the functor $\Phi_K(F) = q_*(K\otimes p^*F)$
has finite cohomological amplitude. Similarly, if $K \in \D^b(X\times Y)$ has finite $\Ext$-amplitude over $Y$
then the functor $\Phi^!_K(G) = q_*\RCHom(K,q^!G)$ has finite cohomological amplitude.
\end{lemma}
\begin{proof}
It suffices to note that the pushforward functor has finite cohomological amplitude
(it is equal to $(0,d)$ where $d$ is the maximum of the dimensions of fibers).
\end{proof}

\subsection{Homotopy colimits}\label{sshocolim}

Recall (see~\cite{BN}) the definition of homotopy colimits in triangulated categories.
Let $F_1 \to F_2 \to F_3 \to \dots$ be a sequence of objects of a triangulated category having countable direct sums.
Its {\sf homotopy colimit}, $\hocolim F_i$, is defined as a cone of the canonical morphism
$\xymatrix@1{\oplus F_i \ar[rr]^{\id-\text{\sf shift}} && \oplus F_i}$,
where $\text{\sf shift}$ denotes the map $\oplus F_i \to \oplus F_i$ defined on $F_i$ as
the composition $F_i \to F_{i+1} \subset \oplus F_j$.
Thus we have a distinguished triangle
$$
\xymatrix@1{\bigoplus F_i\ \ar[rr]^{\id-\text{\sf shift}} && \ \bigoplus F_i\ \ar[rr] && \ \hocolim F_i}.
$$
In what follows we only consider homotopy colimits over the set of positive integers.
Colimits over other partially ordered sets are not considered at all.



\begin{lemma}\label{hclc}
If a functor $\Phi$ commutes with countable direct sums, that is the canonical
morphism $\oplus_i \Phi(F_i) \to \Phi(\oplus_i F_i)$ is an isomorphism then
$\Phi$ commutes with homotopy colimits in the sense that there is
a noncanonical isomorphism $\hocolim\Phi(F_i) \cong \Phi(\hocolim F_i)$.
In particular, homotopy colimits commute with tensor products, pullbacks and pushforwards.
\end{lemma}
\begin{proof}
By the assumptions we have a diagram
$$
\xymatrix{
\bigoplus_i \Phi(F_i) \ar[rr]^{\id-\text{\sf shift}} \ar[d]^{\cong} && \bigoplus_i \Phi(F_i) \ar[rr] \ar[d]^{\cong} && \hocolim \Phi(F_i) \\
\Phi(\bigoplus_i F_i) \ar[rr]^{\id-\text{\sf shift}} && \Phi(\bigoplus_i F_i) \ar[rr] && \Phi(\hocolim F_i)
}
$$
which is evidently commutative. It follows that there is an isomorphism $\hocolim\Phi(F_i) \cong \Phi(\hocolim F_i)$.
For the second claim we use the fact that countable direct sums commute with tensor products, pullbacks (evident)
and pushforwards (\cite{BV}, 3.3.4).
\end{proof}

\begin{remark}
Note that by~\cite{BV} 3.3.4 tensor products, pullbacks and pushforward commute with arbitrary direct sums (not only with countable).
We will use subsequently this fact.
\end{remark}

Now assume that the triangulated category under consideration is the unbounded derived category $\D(\CA)$,
where $\CA$ is an abelian category with exact countable colimits.

\begin{lemma}\label{tlm}
If $F_1 \to F_2 \to F_3 \to \dots$ is a direct system of complexes in $\CA$
and $F$ is the complex obtained by taking termwise colimits of the above direct system
then $\hocolim F_i \cong F$.
\end{lemma}
\begin{proof}
Consider the sequence of complexes $\xymatrix@1{\oplus F_i \ar[rr]^{\id-\text{\sf shift}} && \oplus F_i \ar[r] & F}$.
Since $\CA$ has exact colimits, it is termwise exact. Therefore $F$ is isomorphic to the cone of
the map $\xymatrix@1{\oplus F_i \ar[rr]^{\id-\text{\sf shift}} && \oplus F_i}$.
\end{proof}

\begin{lemma}[\cite{BN}]\label{hlim}
If $\{F_i\}$ is a direct system in $\D(\CA)$ then we have $\CH^n(\hocolim F_i) \cong \dlim \CH^n(F_i)$.
\end{lemma}
\begin{proof}
The long exact sequence of cohomology sheaves of the triangle defining $\hocolim F_i$ gives
$$
\dots \to \oplus_i \CH^n(F_i) \to \oplus_i \CH^n(F_i) \to \CH^n(\hocolim F_i) \to \oplus_i \CH^{n+1}(F_i) \to \oplus_i \CH^{n+1}(F_i) \to \dots
$$
Since the category $\CA$ has exact colimits the last map above is injective.
It follows that $\CH^n(\hocolim F_i) \cong \Coker (\oplus_i \CH^n(F_i) \to \oplus_i \CH^n(F_i)) \cong \dlim \CH^n(F_i)$,
the last isomorphism being the definition of the colimit.
\end{proof}

\begin{lemma}\label{mhcl}
If $\{F_i\}$ is a direct system and there is given a morphism of this direct system to $F$
then there exists a map $\hocolim F_i \to F$ compatible with the maps $F_i \to F$.
Moreover, if $\dlim \CH^t(F_i) = \CH^t(F)$ for each $t \in \ZZ$ then $\hocolim F_i \cong F$.
\end{lemma}
\begin{proof}
We have a canonical map $\oplus F_i \to F$. Its composition with
$\xymatrix@1{\oplus F_i \ar[rr]^{\id-\text{\sf shift}} && \oplus F_i}$ vanishes
since the map is induced by a map of the direct system $\{F_i\}$ to $F$.
Hence it can be factored through a map $\hocolim F_i \to F$. On the $t$-th cohomology
it gives the map $\CH^t(\hocolim F_i) = \dlim\CH^t(F_i) \to \CH^t(F)$ induced
by the map of the direct system $\{\CH^t(F_i)\}$ to $\CH^t(F)$. If it is an isomorphism
for all $t$ then the map $\hocolim F_i \to F$ is a quasiisomorphism.
\end{proof}

\subsection{Approximation}\label{ss_appr}

We say that a direct system $\{F_i\}$ in $\D(\CA)$ {\sf approximates $F \in \D(\CA)$} if there is given
a morphism from the direct system to $F$ such that for any $n \ge 0$ the map
$\tau^{\le n}\tau^{\ge -n}F_k \to \tau^{\le n}\tau^{\ge -n}F$ is an isomorphism
for $k \gg 0$. The following is an immediate corollary of Lemma~\ref{mhcl}.

\begin{lemma}\label{appr}
If a direct system $\{F_i\}$ approximates $F$ in $\D(\CA)$ then $\hocolim F_i \cong F$.
\end{lemma}

Recall (see~\cite{K3}) that a direct system $\{F_i\}$ in $\D(\CA)$ is said to be {\sf stabilizing in finite degrees}\/
if for any $n \in \ZZ$ the map $\tau^{\ge n}F_i \to \tau^{\ge n}F_{i+1}$ is an isomorphism
for $i \gg 0$.

Let $\CB \subset \CA$ be an abelian subcategory and let $\D_\CB^-(\CA)$ denote the full subcategory
in $\D^-(\CA)$, the bounded above derived category of $\CA$,  consisting of all objects with cohomology in $\CB$.

\begin{lemma}\label{sfd}
If a direct system $\{F_i\}$ in $\D^-_\CB(\CA)$ stabilizes in finite degrees then $\hocolim F_i \in \D^-_\CB(\CA)$.
\end{lemma}
\begin{proof}
Follows immediately from Lemma~\ref{hlim}.
\end{proof}

The following easy Lemma shows that every object of $\D^-(X)$ can be approximated
by a stabilizing in finite degrees direct system of perfect complexes. This fact will
be used subsequently in the paper.

\begin{lemma}\label{dmsfd}
For every $F \in \D^-(X)$ there is a stabilizing in finite degrees direct system
of perfect complexes $F_k \in \D^\perf(X)$ which approximates $F$.
In particular, $\hocolim F_k \cong F$.
\end{lemma}
\begin{proof}
Choose a locally free resolution for $F$ and denote by $F_k$ its stupid truncation at degree $-k$.
Then $F_k$ is a perfect complex and $F_k$ form a stabilizing in finite degrees direct system.
Moreover, for any $n\in\ZZ$ we have $\tau^{\ge -n}F_k \cong \tau^{\ge -n}F$ for $k \gg 0$,
hence $F_k$ approximates $F$. By Lemma~\ref{appr} we have $F \cong \hocolim F_k$.
\end{proof}

%

We are also interested in approximation of arbitrary unbounded quasicoherent complexes.
Certainly arbitrary objects of $\D_{qc}(X)$ can't be represented as homotopy colimits
of perfect complexes. There is however the following implicit approximation result.

\begin{lemma}\label{perf-gen-qc}
The minimal full triangulated subcategory of $\D_{qc}(X)$ closed under arbitrary direct sums
and containing $\D^\perf(X)$ is $\D_{qc}(X)$.
%
\end{lemma}
\begin{proof}
Let $\CR \subset \D_{qc}(X)$ be the minimal full triangulated subcategory closed
under arbitrary direct sums and containing $\D^\perf(X)$. By the Bousfield localization
Theorem (see~\cite{N1}, Lemma 1.7) there is a semiorthogonal decomposition $\D_{qc}(X) = \langle \CR^\perp, \CR \rangle$
(the category $\CR^\perp$ is the category of $\CR$-local objects). But $\CR^\perp \subset (\D^\perf(X))^\perp$
and the latter category is zero (e.g.~by the argument in the proof of Lemma~\ref{slinso}), hence
$\CR = \D_{qc}(X)$.
\end{proof}

%
%
%

We conclude this section with the following simple result which will be used later.

\begin{lemma}\label{pfl}
Let $\phi:Y \to X$ be a quasiprojective morphism and assume that $L$ is a line bundle on $Y$ ample over $X$.
If $F \in \D^{[p,q]}(Y)$ then for any $k \gg 0$ there is a direct system $G_m$ in $\D^{[p,q]}(X)$ such that
$\phi_*(F\otimes L^k) \cong \hocolim G_m$.
\end{lemma}
\begin{proof}
Taking the smart truncations of $F$ at $p$ and $q$ we can assume that $F$ is a complex such that $F^t = 0$
unless $t \in [p,q]$. Since $L$ is ample over $X$ for $k \gg 0$ the higher direct images of $F^t\otimes L^k$
vanish hence $\phi_*(F\otimes L^k)$ is isomorphic to the complex
$$
\dots \to 0 \to R_0\phi_*(F^p\otimes L^k) \to \dots \to R_0\phi_*(F^q\otimes L^k) \to 0 \to \dots
$$
Since $\phi$ is quasiprojective, each sheaf $R_0\phi_*(F^t\otimes L^k)$ is a quasicoherent sheaf
which can represented as a countable union of coherent subsheaves. Choose such representation
$R_0\phi_*(F^t\otimes L^k) = \cup C^t_i$ and take
$$
G^t_m = \cup_{i \le m} C^t_i + d(\cup_{i\le m} C^{t-1}_i).
$$
Then it is clear that $G_m$ form a direct system of complexes the termwise colimit of which is the above complex.
Hence $\phi_*(F\otimes L^k) \cong \hocolim G_m$ by Lemma~\ref{tlm}.
\end{proof}

\section{Inducing a semiorthogonal decomposition}

Let $\CT$ and $\CT'$ be triangulated categories and assume that we are given semiorthogonal decompositions
$\CT = \lan \CA_1,\dots,\CA_m \ran$ and $\CT' = \lan \CA'_1,\dots,\CA'_m \ran$.
A triangulated functor $\Phi:\CT \to \CT'$ is {\sf compatible}\/
with the semiorthogonal decompositions if $\Phi(\CA_i) \subset \CA'_i$ for all $1\le i\le m$.

Let $\alpha_i:\CT \to \CT$ and $\alpha'_i:\CT' \to \CT'$ be the projection functors
of the semiorthogonal decompositions.

\begin{lemma}\label{cc}
If the functor $\Phi$ is compatible with the semiorthogonal decompositions then it commutes
with the projection functors, that is we have an isomorphism of functors
$\Phi\circ\alpha_i \cong \alpha'_i\circ\Phi$.
\end{lemma}
\begin{proof}
Take any $T \in \CT$ and let
$$
\xymatrix@C=5pt{
0 \ar@{=}[rr] && T_m \ar[rr] && T_{m-1} \ar[rr] \ar[dl] && T_{m-2} \ar[rr] \ar[dl] && \dots \ar[rr] && T_2 \ar[rr] && T_1 \ar[rr] \ar[dl] && T_0 \ar@{=}[rr] \ar[dl] && T \\
&&& \alpha_m(T) \ar@{..>}[ul] && \alpha_{m-1}(T) \ar@{..>}[ul] &&& \dots &&& \alpha_2(T) \ar@{..>}[ul] && \alpha_1(T) \ar@{..>}[ul]
}
$$
be the filtration of $T$ with factors in $\CA_i$. Applying the functor $\Phi$ we obtain a diagram
$$
\xymatrix@C=-8pt{
\hbox to 0 pt {\hss$0 = $}\Phi(T_m) \ar[rr] && \Phi(T_{m-1}) \ar[rr] \ar[dl] && \Phi(T_{m-2})  \ar[dl] && \to \dots \to
 && \Phi(T_2) \ar[rr] && \Phi(T_1) \ar[rr] \ar[dl] && \Phi(T_0) \hbox to 0 pt {$ = \Phi(T)$\hss} \ar[dl] \\
& \Phi(\alpha_m(T)) \ar@{..>}[ul] && \Phi(\alpha_{m-1}(T)) \ar@{..>}[ul] &&& \dots &&& \Phi(\alpha_2(T)) \ar@{..>}[ul] && \Phi(\alpha_1(T)) \ar@{..>}[ul]
}
$$
Since $\Phi(\alpha_i(T)) \in \Phi(\CA_i) \subset \CA'_i$ we see that this diagram
gives the filtration of $\Phi(T)$ with factors in $\CA'_i$, hence we get isomorphisms
$\Phi(\alpha_i(T)) \cong \alpha'_i(\Phi(T))$. Since such filtration is functorial by Lemma~\ref{ff},
the obtained isomorphisms are functorial as well.
\end{proof}

\begin{lemma}\label{r}
Assume that $\CT = \lan \CA_1,\dots,\CA_m \ran$ and $\CT = \lan \CA'_1, \dots, \CA'_m \ran$
are semiorthogonal decompositions such that $\CA'_i \subset \CA_i$ for all $1 \le i \le m$.
Then $\CA'_i = \CA_i$ for all $i$.
\end{lemma}
\begin{proof}
The identity functor $\CT \to \CT$ is compatible with these semiorthogonal decompositions, hence
their projections functors are isomorphic by Lemma~\ref{cc}. In particular,
for any $i$ and any $A \in \CA_i$ we have $A \cong \alpha_i(A) \cong \alpha'_i(A) \in \CA'_i$,
where $\alpha_i$ and $\alpha'_i$ are the projection functors,
hence $\CA_i \subset \CA'_i$.
\end{proof}

\begin{lemma}\label{un}
If $\Phi:\CT \to \CT'$ is a fully faithful functor and $\CT' = \lan \CA'_1, \dots, \CA'_m\ran$
is a semiorthogonal decomposition then there exists at most one semiorthogonal decomposition
of $\CT$ compatible with $\Phi$, which is given by $\CA_i = \Phi^{-1}(\CA'_i)$.
\end{lemma}
\begin{proof}
Let $\CT = \lan \CA_1, \dots, \CA_m \ran$ be a semiorthogonal decomposition compatible with $\Phi$.
Then we have $\CA_i \subset \Phi^{-1}(\CA'_i)$. On the other hand, let $A \in \Phi^{-1}(\CA'_i)$.
Then $\alpha'_j(\Phi(A)) = 0$ for all $j \ne i$. Hence by Lemma~\ref{cc} we have $\Phi(\alpha_j(A)) = 0$ for all $j \ne i$.
But since $\Phi$ is fully faithful, it follows that $\alpha_j(A) = 0$ for all $j \ne i$, so $A \in \CA_i$.
Thus we are forced to have $\CA_i = \Phi^{-1}(\CA'_i)$.
\end{proof}


In general the collection of subcategories $\CA_i = \Phi^{-1}(\CA'_i)$
does not give a semiorthogonal decomposition. Actually, it is easy to see that
this collection is semiorthogonal (by faithfulness of $\Phi$), however it
can be not full. The simplest example is the functor $\Phi:\D^b(\kk) \to \D^b(\PP^1)$
which takes $\kk$ to $\CO_{\PP^1}$. If one considers the semiorthogonal decomposition
$\D^b(\PP^1) = \lan \CA'_1,\CA'_2 \ran$ with $\CA'_i = \lan \CO_{\PP^1}(i) \ran$
then $\Phi^{-1}(\CA'_i) = 0$ for $i = 1,2$.

Nevertheless, if the subcategories $\CA_i = \Phi^{-1}(\CA'_i)$ form a semiorthogonal decomposition of $\CT$
we will say that this decomposition is {\sf induced}\/ on $\CT$ by the semiorthogonal decomposition of $\CT'$
via $\Phi$.

\begin{lemma}\label{phist}
Let $\Phi:\CT \to \CT'$ be a fully faithful functor and $\CT' = \lan \CA'_1, \dots, \CA'_m\ran$
a semiorthogonal decomposition. It induces a semiorthogonal decomposition on $\CT$ if and only if
the image of $\Phi$ is stable under the projection functors of the semiorthogonal decomposition of $\CT'$.
\end{lemma}
\begin{proof}
The ``only if'' part follows immediately from Lemma~\ref{cc}. For the if part we only have to prove
that every object $T$ of $\CT$ can be decomposed with respect to the collection of subcategories
$\CA_i = \Phi^{-1}(\CA'_i)$. So, let $T' = \Phi(T)$ and let $0 = T'_m \to T'_{m-1} \to \dots \to T'_1 \to T'_0 = T'$
be its filtration with factors in $\CA'_i$. Note that the factors are given by
$\alpha'_i(T') \cong \alpha'_i(\Phi(T))$. Since the image of $\Phi$ is stable under $\alpha'_i$,
it follows that $\alpha'_i(T') \cong \Phi(A_i)$ for some objects $A_i \in \CA_i$.
Let us check that these are the components of $T$. To do this we have to construct a filtration
$0 = T_m \to T_{m-1} \to \dots \to T_1 \to T_0 = T$ such that its factors are isomorphic to $A_i$.
We do it inductively. First of all, we put $T_0 = T$. Now assume that $T_i$ is constructed in such a way
that $\Phi(T_i) \cong T'_i$. Then we compose this isomorphism with the map $T'_i \to \alpha'_i(T') \cong \Phi(A_i)$.
Since $\Phi$ is fully faithful, the resulted map comes from a map $T_i \to A_i$ in $\CT$.
We take $T_{i+1}$ to be the cone of this morphism shifted by $-1$. Applying to the triangle $T_{i+1} \to T_i \to A_i$
the functor $\Phi$ we conclude that $\Phi(T_{i+1}) \cong T'_{i+1}$. Applying this procedure $m$ times
we construct $T_m$. Note that $\Phi(T_m) \cong T'_m = 0$. Since $\Phi$ is fully faithful, it follows that
$T_m = 0$, so the desired filtration of $T$ is constructed.
\end{proof}

\begin{lemma}\label{endo}
Let $\Phi:\CT \to \CT'$ be a fully faithful embedding, and assume that
$\CT' = \lan \CA'_1, \dots, \CA'_m\ran$ and $\CT = \lan \CA_1, \dots, \CA_m \ran$
are semiorthogonal decompositions compatible with $\Phi$. Let $\Psi':\CT' \to \CT'$
be an endofunctor, such that $\CT$ and all $\CA'_i$ are stable under $\Psi'$.
Then every $\CA_i$ is also stable under $\Psi'$.
\end{lemma}
\begin{proof}
Since $\CT$ is stable under $\Psi'$ and $\Phi$ is fully faithful,
the restriction of $\Psi'$ to $\CT$ defines an endofunctor $\Psi:\CT \to \CT$,
such that $\Phi\circ\Psi = \Psi' \circ\Phi$.
Since $\CA_i = \Phi^{-1}(\CA'_i)$ we have to check that $\Phi(\Psi(\CA_i)) \subset \CA'_i$.
But $\Phi(\Psi(\CA_i)) = \Psi'(\Phi(\CA_i)) \subset \Psi'(\CA'_i) \subset \CA'_i$
since $\CA'_i$ is $\Psi'$-stable.
\end{proof}

\section{Extensions of a semiorthogonal decomposition}

Let $X$ be an algebraic variety and assume that we are given a semiorthogonal decomposition
of $\D^b(X)$. In this section we construct a compatible system of semiorthogonal decompositions
of the categories $\D^\perf(X) \subset \D^-(X) \subset \D_{qc}(X)$.

\subsection{Perfect complexes}

First of all we note that any strong semiorthogonal decomposition (see Defenition~\ref{strongsod}) of $\D^b(X)$
induces a semiorthogonal decomposition of the category of perfect complexes.

\begin{proposition}\label{sodp}
Let $\D^b(X) = \lan \CA_1,\dots,\CA_m \ran$ be a strong semiorthogonal decomposition.
Then there is a unique semiorthogonal decomposition of the category $\D^\perf(X)$
compatible with the natural embedding $\D^\perf(X) \to \D^b(X)$.
\end{proposition}
\begin{proof}
The existence of a semiorthogonal decomposition of $\D^\perf(X)$
compatible with that of $\D^b(X)$ follows from \cite{O2}, 1.10 and 1.11.
Moreover, it follows from Lemma~\ref{un} that the components of this decomposition
are given by
\begin{equation}\label{ap}
\CA_i^\perf = \CA_i \cap \D^\perf(X)
\end{equation}
and that the decomposition is unique.
\end{proof}

\subsection{Unbounded quasicoherent complexes}

Now we are going to show that any (not necessarily strong) semiorthogonal decomposition of $\D^\perf(X)$
induces a semiorthogonal decomposition of the unbounded derived category
of quasicoherent sheaves~$\D_{qc}(X)$.

\begin{proposition}\label{sodqc}
Let $\D^\perf(X) = \lan \CA_1^\perf,\dots,\CA_m^\perf \ran$ be a semiorthogonal decomposition.
Then there is a unique semiorthogonal decomposition $\D_{qc}(X) = \lan \HCA_1,\dots,\HCA_m \ran$
compatible with the natural embedding $\D^\perf(X) \to \D_{qc}(X)$
and with closed under arbitrary direct sums components.
The projection functors $\halpha_i$ of this decomposition commute with direct sums
and homotopy colimits.

Moreover, if the initial decomposition of the category $\D^\perf(X)$ is induced by a semiorthogonal decomposition
$\D^b(X) = \lan \CA_1,\dots,\CA_m \ran$ of $\D^b(X)$ the projection functors of which
have finite right cohomological amplitude then the obtained decomposition
of $\D_{qc}(X)$ is compatible with the natural embedding $\D^b(X) \to \D_{qc}(X)$ as well.
\end{proposition}
\begin{proof}
Define the subcategory $\HCA_i \subset\D_{qc}(X)$ to be the subcategory of $\D_{qc}(X)$ obtained
by iterated addition of cones to the closure of $\CA_i^\perf$ in $\D_{qc}(X)$ under all direct sums.
Let us check that the categories $\HCA_i$ form a semiorthogonal decomposition of $\D_{qc}(X)$.
First of all, if $j > i$,
$A_j^l \in \CA_j^\perf$, $A_i^k \in \CA_i^\perf$ then
$$
\Hom(\oplus_l A_j^l,\oplus_k A_i^k) \cong
\prod_l \Hom(A_j^l,\oplus_k A_i^k) \cong
\prod_l \bigoplus_k \Hom(A_j^l,A_i^k) = 0
$$
(in the second isomorphism we used the fact that $A_j^l$ are perfect complexes, hence compact objects of $\D_{qc}(X)$).
Addition of cones does not spoil semiorthogonality (see Lemma~\ref{lrso}), hence
the collection of subcategories $\HCA_1,\dots,\HCA_m$ is semiorthogonal.
Note also that a direct sum of cones is a cone of direct sums by~\cite{KSch}, 10.1.19,
so $\HCA_i$ is a closed under all direct sums triangulated subcategory of $\D_{qc}(X)$.

Now consider the triangulated subcategory $\lan \HCA_1,\dots,\HCA_m \ran$ generated in $\D_{qc}(X)$
by the subcategories $\HCA_1,\dots,\HCA_m$. It is clear that it is a triangulated subcategory
of $\D_{qc}(X)$ closed under all direct sums. Moreover, it contains $\lan \CA_1^\perf,\dots,\CA_m^\perf \ran = \D^\perf(X)$.
Hence it coincides with $\D_{qc}(X)$ by Lemma~\ref{perf-gen-qc}.
This means that $\lan \HCA_1,\dots,\HCA_m \ran = \D_{qc}(X)$.
The uniqueness of such semiorthogonal decomposition is evident by Lemma~\ref{r}.

The compatibility with the embedding $\D^\perf(X) \to \D_{qc}(X)$ and
closedness under arbitrary direct sums are evident. Commutativity of $\halpha_i$
with arbitrary direct sums follows immediately and for homotopy colimits we apply Lemma~\ref{hclc}.

Further, to check that the constructed semiorthogonal decomposition of $\D_{qc}(X)$ is compatible with
the semiorthogonal decomposition of $\D^b(X)$ we have to check that for any $A \in \CA_i \subset \D^b(X)$
we have $\halpha_i(A) \cong A$. Indeed, choose a locally free resolution $P^\bullet \to A$,
and take $A^n = \sigma^{\ge -n}(P^\bullet)$, the stupid truncation of the complex $P^\bullet$ at degree $-n$,
so that we have a distinguished triangle
$$
\sigma^{\ge -n}P^\bullet \to A \to \sigma^{\le -n-1}P^\bullet.
$$
Note that the direct system $\sigma^{\ge -n}P^\bullet$ approximates $A$ in the sense of paragraph~\ref{ss_appr},
hence by Lemma~\ref{appr} we have an isomorphism $\hocolim (\sigma^{\ge -n}P^\bullet) \cong A$.
Therefore
$$
\halpha_i(A) \cong
\halpha_i(\hocolim (\sigma^{\ge -n}P^\bullet)) \cong
\hocolim \halpha_i(\sigma^{\ge -n}P^\bullet) \cong
\hocolim \alpha_i(\sigma^{\ge -n}P^\bullet),
$$
the last isomorphism is due to the fact that $\sigma^{\ge -n}P^\bullet$ is a perfect complex.
So, it suffices to check that $\hocolim \alpha_i(\sigma^{\ge -n}P^\bullet) \cong A$.
Indeed, applying $\alpha_i$ to the above triangle we obtain
$$
\alpha_i(\sigma^{\ge -n}P^\bullet) \to A \to \alpha_i(\sigma^{\le -n-1}P^\bullet).
$$
Let $(a_i,b_i)$ be the cohomological amplitude of the functor $\alpha_i$.
Since $\sigma^{\le -n-1}P^\bullet \in \D^{\le -n-1}(X)$ we have
$\alpha_i(\sigma^{\le -n-1}P^\bullet) \in \D^{\le -n-1+b_i}(X)$,
hence $\alpha_i(\sigma^{\ge -n}P^\bullet)$ approximates $A$, so
$\hocolim \alpha_i(\sigma^{\ge -n}P^\bullet) \cong A$.
%
\end{proof}


\subsection{Bounded above coherent complexes}

The next step is the following.

\begin{proposition}\label{sodm}
Let $\D^\perf(X) = \lan \CA_1^\perf,\dots,\CA_m^\perf \ran$ be a semiorthogonal decomposition.
Then there is a unique semiorthogonal decomposition of $\D^-(X)$ compatible with this decomposition
of $\D^\perf(X)$ and with the decomposition of $\D_{qc}(X)$ constructed in Proposition~\ref{sodqc}
with respect to the natural embeddings $\D^\perf(X) \to \D^-(X) \to \D_{qc}(X)$.
Its components are closed under homotopy colimits of stabilizing in finite degrees direct systems.
\end{proposition}
\begin{proof}
We have to check that $\D^-(X)$ is stable under the projection functors $\halpha_i$.
Then by Lemma~\ref{phist} it would follow that the subcategories
\begin{equation}\label{am}
\CA_i^- = \HCA_i \cap \D^-(X)
\end{equation}
give a semiorthogonal decomposition, which is evidently compatible with those of $\D^\perf(X)$
and $\D_{qc}(X)$. So, we take any $F \in \D^-(X)$. By Lemma~\ref{dmsfd} there exists a stabilizing
in finite degrees direct system of perfect complexes $F_k$ such that $F \cong \hocolim F_k$.
It follows that
$$
\halpha_i(F) \cong
\halpha_i(\hocolim F_k) \cong
\hocolim \alpha_i(F_k)
$$
(the second isomorphism follows from Proposition~\ref{sodqc}).
But by Lemma~\ref{fca} the direct system $\alpha_i(F_k)$ also stabilizes in finite degrees,
so it follows from Lemma~\ref{sfd} that $\hocolim\alpha_i(F_k) \in \D^-(X)$.

The last claim is clear since both $\HCA_i$ and $\D^-(X)$ are closed under
homotopy colimits of stabilizing in finite degrees direct systems.
\end{proof}

\subsection{$S$-linearity}

Assume that $X$ is a scheme over $S$, that is we are given a map $f:X \to S$.
Recall that any strong semiorthogonal decomposition of $\D^b(X)$ by Proposition~\ref{sodp} induces a compatible
semiorthogonal decomposition of $\D^\perf(X)$, which in its turn by Propositions~\ref{sodqc} and \ref{sodm}
induces compatible semiorthogonal decompositions of $\D_{qc}(X)$ and $\D^-(X)$.

\begin{lemma}\label{isl}
If the initial semiorthogonal decomposition of the category $\D^b(X)$ is $S$-linear then
the induced semiorthogonal decomposition of $\D^\perf(X)$ is $S$-linear. Similarly,
if the semiorthogonal decomposition of the category $\D^\perf(X)$ is $S$-linear then
the induced semiorthogonal decompositions of $\D_{qc}(X)$ and $\D^-(X)$ are $S$-linear as well.
\end{lemma}
\begin{proof}
Take any $G \in \D^\perf(S)$. Then $\Psi_G(H) := H\otimes f^*G$ is an endofunctor of $\D_{qc}(X)$
which preserves $\D^-(X)$, $\D^b(X)$ and $\D^\perf(X)$ as well as the initial semiorthogonal
decomposition. It follows from Lemma~\ref{endo} that the semiorthogonal decomposition~\eqref{ap}
of $\D^\perf(X)$ is stable under $\Psi_G$. Now let us check that each component $\HCA_i$
of the semiorthogonal decomposition of $\D_{qc}(X)$ is stable under $\Psi_G$.
Indeed, by definition $\HCA_i$ is the smallest triangulated subcategory of $\D_{qc}(X)$
containing $\CA_i^\perf$ and closed under arbitrary direct sums. But the functor $\Psi_G$
commutes with direct sums (see~\cite{BV}, 3.3.4) and is exact which implies the claim.
Again applying Lemma~\ref{endo} we conclude that the semiorthogonal decomposition~\eqref{am} of $\D^-(X)$
is also stable under $\Psi_G$. Since this is true for all $G \in \D^\perf(S)$,
we see that all these decompositions are $S$-linear.
\end{proof}

Actually, for the components of semiorthogonal decompositions of $\D_{qc}(X)$ and $\D^-(X)$
we have a stronger result.

\begin{lemma}\label{isl1}
If $\D^-(X) = \lan \CA_i^-,\dots,\CA_m^- \ran$ is an $S$-linear semiorthogonal decomposition
with components closed under homotopy colimits of stabilizing in finite degrees direct systems
then $\CA_i^-\otimes f^*\D^-(S) \subset \CA_i^-$. Similarly,
if $\D_{qc}(X) = \lan \HCA_i,\dots,\HCA_m \ran$ is an $S$-linear semiorthogonal decomposition
with components closed under arbitrary direct sums then $\HCA_i\otimes f^*\D_{qc}(S) \subset \HCA_i$.
\end{lemma}
\begin{proof}
Take any $G$ in $\D^-(S)$. Applying Lemma~\ref{dmsfd} choose a stabilizing
in finite degrees direct system of perfect complexes $G_k$ approximating $G$ so that $G \cong \hocolim G_k$.
Then for any $F \in \CA_i^-$ we have $F \otimes f^*G \cong F \otimes f^*(\hocolim G_k) \cong \hocolim(F \otimes f^*G_k)$.
Since the functors $\otimes$ and $f^*$ are right exact, it follows that the direct system $F\otimes f^*G_k$
stabilizes in finite degrees. Hence its homotopy colimit belongs to $\CA_i^-$ since $\CA_i^-$
is $S$-linear and closed under homotopy colimits of stabilizing in finite degrees direct systems.

For the second claim recall that by Lemma~\ref{perf-gen-qc} the category $\D_{qc}(S)$ can be obtained by iterated addition of cones
to the closure of $\D^\perf(S)$ under arbitrary direct sums. Further, we know by Lemma~\ref{isl}
that $\HCA_i\otimes f^*G \subset \HCA_i$ for any perfect $G$. Since $f^*$ and $\otimes$ commute with
direct sums, it follows that the same is true for $G$ being arbitrary direct sum of perfect complexes.
Finally, since $f^*$ and $\otimes$ are exact and $\HCA_i$ is triangulated, the same embedding holds
for arbitrary $G$.
\end{proof}

\section{Change of a base}

Let $f:X \to S$ be an algebraic map. Consider a base change $\phi: T \to S$ and denote by $X_T = X\times_S T$
the fiber product. Denote the projections $X_T \to T$ and $X_T \to X$ by $f$ and $\phi$ respectively,
so that we have a cartesian diagram
\begin{equation}\label{xt}
\vcenter{\xymatrix{
X_T \ar[r]^\phi \ar[d]_f & X \ar[d]^f \\ T \ar[r]^\phi & S
}}
\end{equation}
Throughout this section we assume that the base change $\phi$ is faithful for $f:X \to S$
(see paragraph~\ref{ssfbc} for the definition).

\subsection{Base change for perfect complexes}

Let $\D^\perf(X) = \lan \CA_1^\perf, \dots, \CA_m^\perf \ran$ be an $S$-linear semiortho\-gonal decomposition.
Let $\CA_{iT}^p$ denote the minimal triangulated subcategory of $\D^\perf(X_T)$ closed under taking direct summands
and containing all objects of the form $\phi^*F \otimes f^*G$ with $F \in \CA_i^\perf$, $G \in \D^\perf(T)$:
\begin{equation}\label{aitp}
\CA_{iT}^p = \lan \phi^*\CA_i^\perf \otimes f^*\D^\perf(T) \ran.
\end{equation}
Note that the subcategory $\CA_{iT}^p \subset \D^\perf(X_T)$ is $T$-linear,
since the generating class $\phi^*\CA_i^\perf \otimes f^*\D^\perf(T)$ is $T$-linear,
and the process of adding cones and direct summands preserves $T$-linearity.

\begin{proposition}\label{dp}
We have $\D^\perf(X_T) = \lan \CA_{1T}^p, \dots, \CA_{mT}^p \ran$,
a $T$-linear semiorthogonal decomposition compatible
with the functor $\phi^*:\D^\perf(X) \to \D^\perf(X_T)$.
\end{proposition}
\begin{proof}
%
Because of Lemma~\ref{slinso} and Lemma~\ref{lrso} to verify semiorthogonality
it suffices to check that
$f_*\RCHom(\phi^*F_i\otimes f^*G,\phi^*F_j\otimes f^*G') = 0$ for any $F_i \in \CA^\perf_i$, $F_j \in \CA^\perf_j$
and any $G,G' \in \D^\perf(T)$ if $i > j$. But
$$
f_*\RCHom(\phi^*F_i\otimes f^*G,\phi^*F_j\otimes f^*G') \cong
f_*\phi^*\RCHom(F_i,F_j)\otimes G^*\otimes G' \cong
\phi^*f_*\RCHom(F_i,F_j)\otimes G^*\otimes G' = 0
$$
(for the first isomorphism we use perfectness of $F_i$, $F_j$, $G$ and $G'$,
for the second we use faithfulness of the base change $\phi$,
and for the third --- $S$-linearity of the initial semiorthogonal decomposition of $D^\perf(X)$
and Lemma~\ref{slinso} for it).

It remains to check that the subcategories $\CA_{iT}^p$ generate $\D^\perf(X_T)$.
Take any object $H \in \D^\perf(X_T)$. Then by Lemma~\ref{dpxt} below it can be obtained
by consecutive taking cones and direct summands starting from the collection of objects
$\phi^*F^t\otimes f^* G^t$, where $F^t \in \D^\perf(X)$, $G^t \in \D^\perf(T)$, and $t = 1, \dots, N$.
On the other hand, every object $F^t$ can be decomposed with respect to the semiorthogonal
decomposition $\D^\perf(X) = \lan \CA_1^\perf,\dots,\CA_m^\perf \ran$, in other words,
it can be obtained by consecutive taking cones from a collection of objects $A^t_i \in \CA^\perf_i$,
$i = 1, \dots, m$. It follows that $H$ can be obtained by consecutive taking cones and direct summands
starting from the collection of objects $\phi^*A_i^t\otimes f^* G^t$, and it remains to note
that $\phi^*A_i^t\otimes f^* G^t \in \CA_{iT}^p$ by definition.

The second claim follows immediately from~\eqref{aitp}.
\end{proof}

\begin{lemma}\label{dpxt}
The category $\D^\perf(X_T)$ coincides with the minimal triangulated subcategory of $\D_{qc}(X)$
closed under taking direct summands and containing the class of objects
$\phi^*\D^\perf(X)\otimes f^*\D^\perf(T) := \{\phi^*F \otimes f^*G\ |\ F \in \D^\perf(X),\ G \in \D^\perf(T)\}$.
\end{lemma}
\begin{proof}
Take any object $H \in \D^\perf(X)$ and construct a locally free resolution $P^\bullet \to H$ in which all
sheaves $P^k$ have form $P^k \cong \phi^*F \otimes f^*G$, where $F$ and $G$ are locally free sheaves
on $X$ and $T$ respectively (this can be done since $\phi$ is quasiprojective).
Then its stupid truncation $\sigma^{\ge n}(P^\bullet) \in \lan \phi^*\D^\perf(X)\otimes f^*\D^\perf(T) \ran$
for all $n$, and for $n \ll 0$ the object $H$ is a direct summand of $\sigma^{\ge n}(P^\bullet)$.
Indeed, since $H$ is a perfect complex it is quasiisomorphic to a bounded complex of locally free sheaves
of finite rank. Assume that this complex is bounded from the left by degree $l \in \ZZ$.
Take $n \le l - \dim X$ and consider the triangle
$$
\sigma^{\ge n}P^\bullet \to P^\bullet \to \sigma^{\le n-1}P^\bullet.
$$
Note that since $P^\bullet$ is quasiisomorphic to $H$ and $H$ is quasiisomorphic
to a complex of locally free sheaves supported in degrees $\ge l$
it follows that the complex computing $\CExt^i(P^\bullet,\sigma^{\le n-1}P^\bullet)$
is supported in degrees $\le n-1-l$. The hypercohomology sequence then shows that
$\Ext^i(P^\bullet,\sigma^{\le n-1}P^\bullet) = 0$ for $i > n-1-l+\dim X$.
But $n-1-l+\dim X \le -1$ for $n \le l - \dim X$, hence $\Hom(P^\bullet,\sigma^{\le n-1}P^\bullet) = 0$.
In particular, the above triangle splits, hence $P^\bullet$ is a direct summand of
$\sigma^{\ge n}P^\bullet$ and we are done since $P^\bullet$ is quasiisomorphic to $H$.
\end{proof}

\subsection{Base change for unbounded quasicoherent complexes}

We start with an $S$-linear semiorthogonal decomposition $\D^\perf(X) = \lan \CA_1^\perf, \dots, \CA_m^\perf \ran$.
Let $\D^\perf(X_T) = \lan \CA_{1T}^p, \dots, \CA_{mT}^p \ran$ be the $T$-linear semiorthogonal decomposition
constructed in Proposition~\ref{dp}.
Then using Proposition~\ref{sodqc} we construct semiorthogonal decompositions 
$\D_{qc}(X) = \lan \HCA_1, \dots, \HCA_m \ran$ and $\D_{qc}(X_T) = \lan \HCA_{1T}, \dots, \HCA_{mT} \ran$.
By Lemma~\ref{isl} these decompositions are $S$ and $T$-linear.

\begin{proposition}\label{dqc}
The functors $\phi_*:\D_{qc}(X_T) \to \D_{qc}(X)$ and $\phi^*:\D_{qc}(X) \to \D_{qc}(X_T)$
are compatible with the above semiorthogonal decompositions. Moreover,
\begin{equation}\label{hait}
\HCA_{iT} = \{ H \in \D_{qc}(X_T)\ |\ \phi_*(H \otimes f^*G) \in \HCA_i\ \text{ for all $G \in \D^\perf(T)$}\}.
\end{equation}
\end{proposition}
\begin{proof}
Recall that both $\HCA_i$ and $\HCA_{iT}$ are obtained from $\HCA_i^\perf$ and $\HCA_{iT}^p$
by addition of arbitrary direct sums and iterated addition of cones and both are closed under
arbitrary direct sums triangulated categories. Since both $\phi_*$ and $\phi^*$ commute
with arbitrary direct sums and are exact, it suffices to check that $\phi^*(\CA_i^\perf) \subset \HCA_{iT}$
and that $\phi_*(\CA_{iT}^p) \subset \HCA_i$. The first is evident by definition of $\HCA_{iT}$.
For the second take any $F \in \CA_i^\perf$, $G \in \D^\perf(T)$. Then
$\phi_*(\phi^*F \otimes f^*G) \cong F \otimes \phi_*f^*G \cong F \otimes f^*\phi_*G$.
But $F \otimes f^*\phi_*G \in \HCA_i$ by Lemma~\ref{isl1}.

To prove~\eqref{hait} we note that the LHS is contained in the RHS by the $T$-linearity of $\HCA_{iT}$
and compatibility with $\phi_*$. Conversely, assume that $H$ is in the RHS but not in $\HCA_{iT}$
so that $\halpha_{jT}(H) \ne 0$ for some $j$. Since the semiorthogonal decomposition
$\lan \HCA_{1T}, \dots, \HCA_{mT} \ran$ is $T$-linear, the functors $\halpha_{jT}$ are $T$-linear by Lemma~\ref{prfsl},
hence $\halpha_{jT}(H\otimes f^*L^k) \cong \halpha_{jT}(H) \otimes f^*L^k$ for any line bundle $L$ on $T$ and any $k \in \ZZ$.
By Lemma~\ref{dbdm} below we have $\hocolim \phi_*(\halpha_{jT}(H)\otimes f^*L^{k_i}) \ne 0$
for some sequence $L^{k_1} \to L^{k_2} \to L^{k_3} \to \dots$ if $L$ is ample over $S$.
It remains to note that
$$
\halpha_j(\hocolim \phi_*(H \otimes f^*L^{k_i})) \cong
\hocolim \halpha_j(\phi_*(H \otimes f^*L^{k_i})) \cong
\hocolim \phi_*(\halpha_{jT}(H \otimes f^*L^{k_i})) \ne 0
$$
(the first isomorphism is by Proposition~\ref{sodqc}, the second is by Lemma~\ref{cc})
so $\hocolim \phi_*(H \otimes f^*L^{k_i}) \not \in \HCA_i$.
But this means that $\phi_*(H \otimes f^*L^k) \not \in \HCA_i$ for some $k \in \ZZ$ since
$\HCA_i$ is closed under homotopy colimits. So, $H$ is not in the RHS of~\eqref{hait},
a contradiction.
\end{proof}

\begin{lemma}\label{dbdm}
Let $\phi: Y \to X$ be a quasiprojective morphism and let $L$ be a line bundle on $Y$ ample over~$X$.
Let $F \in \D_{qc}(Y)$. Then $F \in \D^{[p,q]}_{qc}(Y)$ if and only if
for any sequence of maps $L^{k_1} \to L^{k_2} \to L^{k_3} \to \dots$ with $k_i \to \infty$ we have
$\hocolim \phi_*(F\otimes L^{k_i}) \in \D^{[p,q]}_{qc}(X)$.
In particular $F = 0$ if and only if for any sequence
$L^{k_1} \to L^{k_2} \to L^{k_3} \to \dots$ with $k_i \to \infty$ we have
$\hocolim \phi_*(F\otimes L^{k_i}) = 0$.
\end{lemma}
\begin{proof}
As $\phi$ is quasiprojective we can represent $\phi$ as $\pi_1\circ j_1$,
where $j_1:Y \to \BY$ is an open embedding and $\pi_1:\BY \to X$ is a projective morphism.
Furthermore, any open embedding $j_1:Y \to \BY$ can be represented as a composition
of an affine open embedding $j:Y \to \TY$ and of a projective morphism $\pi_2:\TY \to \BY$
(we take for $\TY$ the blowup of the ideal of the closed subset $\BY\setminus Y$ in $Y$).
Put $\pi = \pi_1\circ\pi_2$. Thus $\phi = \pi\circ j$, where $j$ is an affine open embedding and $\pi$ is projective.
Since $j$ is an affine open embedding the functors $j_*$ and $j^*$ are exact and $j^*j_* \cong \id$,
hence we have $F \in \D^{[p,q]}_{qc}(Y)$ if and only if $j_*F \in \D^{[p,q]}_{qc}(\TY)$.
Thus the claim of the Lemma reduces to the case when $\phi$ is projective.

So, assume that $\phi$ is projective.
For any nonzero coherent sheaf $H$ on $X$ we know that $\CH^t(\phi_*(H\otimes L^k))$ is zero for $t\ne 0$ and $k \gg 0$.
Therefore for any quasicoherent sheaf $H$ on $X$ we have $\dlim \CH^t(\phi_*(H\otimes L^{k_i})) = 0$ for $t\ne 0$ if $k_i \to \infty$.
So, the hypercohomology spectral sequence and Lemma~\ref{hlim} imply that
$$
\CH^t(\hocolim \phi_*(F \otimes L^{k_i})) \cong \dlim \CH^0(\phi_*(\CH^t(F)\otimes L^{k_i})).
$$
It follows immediately that $F \in \D^{[p,q]}_{qc}(Y)$ implies $\hocolim \phi_*(F\otimes L^{k_i}) \in \D^{[p,q]}_{qc}(X)$.
As for the other implication it suffices to check that for any quasicoherent sheaf $H \ne 0$ on $Y$
there exists a sequence of maps $L^{k_1} \to L^{k_2} \to L^{k_3} \to \dots$ with $k_i \to \infty$
such that $\dlim \CH^0(\phi_*(H \otimes L^{k_i})) \ne 0$. Since tensoring with a line bundle and the colimit are exact
functors on the abelian category $\Qcoh(X)$, while $\CH^0\phi_*$ is left exact, it follows that it suffices
to prove the above for any nonzero subsheaf of $H$. Thus we can assume that $H$ is coherent.
Then using ampleness of $L$ we can find $m$ and a section $s$ of $L^m$ such that the map
$H \to H\otimes L^m$ given by $s$ is an embedding. Now consider the sequence $L^m \to L^{2m} \to L^{3m} \to \dots$
with all maps given by~$s$. Then all the maps in the sequence
$\CH^0(\phi_*(H\otimes L^{m})) \to \CH^0(\phi_*(H\otimes L^{2m})) \to \CH^0(\phi_*(H\otimes L^{3m})) \to \dots$
are embeddings. Moreover, $\CH^0(\phi_*(H\otimes L^{im})) \ne 0$ for $i \gg 0$. Hence the limit is nonzero and we are done.
\end{proof}

\subsection{Base change for bounded above coherent complexes}

As above we start with an $S$-linear semiorthogonal decomposition
$\D^\perf(X) = \lan \CA_1^\perf, \dots, \CA_m^\perf \ran$.
Let $\D^\perf(X_T) = \lan \CA_{1T}^p, \dots, \CA_{mT}^p \ran$ be the $T$-linear semiorthogonal decomposition
constructed in Proposition~\ref{dp}.
Let $\D_{qc}(X) = \lan \HCA_1, \dots, \HCA_m \ran$
and $\D_{qc}(X_T) = \lan \HCA_{1T}, \dots, \HCA_{mT} \ran$ be the $S$ and $T$-linear semiorthogonal decompositions
constructed in Proposition~\ref{sodqc} from the above decompositions of $\D^\perf(X)$ and $\D^\perf(X_T)$ respectively.
Finally, let $\D^-(X) = \lan \CA_1^-,\dots,\CA_m^- \ran$ and $\D^-(X_T) = \lan \CA_{1T}^-,\dots,\CA_{mT}^- \ran$
be the $S$ and $T$-linear semiorthogonal decompositions constructed in Proposition~\ref{sodm}.

\begin{lemma}\label{dm}
The functors $\phi_*:\D^-(X_T) \to \D_{qc}(X)$ and $\phi^*:\D^-(X) \to \D^-(X_T)$
are compatible with the above semiorthogonal decompositions.
\end{lemma}
\begin{proof}
Follows immediately from Proposition~\ref{dqc} since $\CA_i^- = \HCA_i \cap \D^-(X)$ and $\CA_{iT}^- = \HCA_{iT} \cap \D^-(X_T)$.
\end{proof}

\subsection{Base change for bounded coherent complexes}

This time we start with an $S$-linear strong semiorthogonal decomposition
$\D^b(X) = \lan \CA_1,\dots,\CA_m \ran$.
Let $\D^\perf(X) = \lan \CA_1^\perf, \dots, \CA_m^\perf \ran$ be the induced $S$-linear semiorthogonal decomposition of $\D^\perf(X)$.
Further, consider the $T$-linear semiorthogonal decomposition
$\D^\perf(X_T) = \lan \CA_{1T}^p, \dots, \CA_{mT}^p \ran$ of Proposition~\ref{dp}, and
let $\D_{qc}(X) = \lan \HCA_1, \dots, \HCA_m \ran$
and $\D_{qc}(X_T) = \lan \HCA_{1T}, \dots, \HCA_{mT} \ran$ be the $S$ and $T$-linear semiorthogonal decompositions
constructed in Proposition~\ref{sodqc} from the above decompositions of $\D^\perf(X)$ and $\D^\perf(X_T)$ respectively.
Further, let $\D^-(X) = \lan \CA_1^-,\dots,\CA_m^- \ran$ and $\D^-(X_T) = \lan \CA_{1T}^-,\dots,\CA_{mT}^- \ran$
be the $S$ and $T$-linear semiorthogonal decompositions constructed in Proposition~\ref{sodm}.
Finally, we define
\begin{equation}\label{ait}
\CA_{iT} = \CA_{iT}^- \cap \D^b(X_T).
\end{equation}

\begin{theorem}\label{sodxt}
Let $\D^b(X) = \lan \CA_1,\dots,\CA_m \ran$ be an $S$-linear strong semiorthogonal decomposition
the projection functors of which have finite cohomological amplitude
and assume that the base change $\phi$ is faithful for~$f$.
Then the subcategories $\CA_{iT} \subset \D^b(X_T)$ defined in~\eqref{ait} form
a $T$-linear semiorthogonal decomposition $\D^b(X_T) = \lan \CA_{1T},\dots,\CA_{mT} \ran$.
The projection functors of this semiorthogonal decomposition have the same cohomological amplitude
as the projection functors of the initial semiorthogonal decomposition.
Moreover, the functors $\phi_*:\D^b(X_T) \to \D_{qc}(X)$ and $\phi^*:\D^b(X) \to \D^-(X_T)$
are compatible with the semiorthogonal decompositions of $\D_{qc}(X)$ and $\D^-(X_T)$ respectively.
\end{theorem}
\begin{proof}
Take any $H \in \D^{[p,q]}(X_T)$.
We have to check that $\alpha_{iT}^-(H)$ is bounded.
Let $(a_i,b_i)$ be the cohomological amplitude of~$\alpha_i$.
Let us show that $\alpha_{iT}^-(H) \in \D^{[p+a_i,q+b_i]}(X_T)$.
This will prove both that the categories $\CA_{iT}$ form a semiorthogonal decomposition
of $\D^b(X_T)$ and that the cohomological amplitude of the projection functors is the same
as that of $\alpha_i$.
Using Lemma~\ref{dbdm} we see that it suffices to check that for $k \gg 0$ we have
$\phi_*(\alpha_{iT}^-(H)\otimes L^k) \in \D^{[p+a_i,q+b_i]}_{qc}(X)$,
where $L$ is a line bundle on $X_T$ ample over $X$. We can take $L = f^*M$ where $M$ is a line bundle on $T$
ample over~$S$. Note that
$\phi_*(\alpha_{iT}^-(H)\otimes f^*M^k) \cong
\phi_*(\alpha_{iT}^-(H\otimes f^*M^k)) \cong
\halpha_i(\phi_*(H\otimes f^*M^k))$
by Lemma~\ref{dm} and Lemma~\ref{fca}.
Further, note that by Lemma~\ref{pfl} for $k \gg 0$ we have
$\phi_*(H\otimes f^*M^k) \cong \hocolim G_m$ for a certain direct system $G_m$
with $G_m \in \D^{[p,q]}(X)$. Therefore
$$
\halpha_i(\phi_*(H\otimes f^*M^k)) =
\halpha_i(\hocolim G_m) \cong
\hocolim \alpha_i(G_m)
$$
since $\halpha_i$ commutes with homotopy colimits.
Finally, $\alpha_i(G_m) \in \D^{[p+a_i,q+b_i]}(X)$, hence
$\hocolim \alpha_i(G_m) \in \D^{[p+a_i,q+b_i]}(X)$ by Lemma~\ref{hlim},
hence $\halpha_i(\phi_*(H\otimes f^*M^k)) \in \D^{[p+a_i,q+b_i]}(X)$
as it was required.

Finally, it remains to check that the subcategories~\eqref{ait} are $T$-linear,
and also that $\phi_*(\CA_{iT}) \subset \HCA_i$ and $\phi^*(\CA_i) \in \CA_{iT}^-$.
The first is clear since $\CA_{iT}^-$ is $T$-linear and the other two claims
follow from Lemma~\ref{dm}.
\end{proof}

The semiorthogonal decomposition of $\D^b(X_T)$ constructed in Theorem~\ref{sodxt} will be referred to
as the induced decomposition of $\D^b(X_T)$ with respect to the base change $\phi$.
Note that the definition of its component $\CA_{iT}$ depends only on $\CA_i$
(i.e.\ doesn't depend on the choice of a semiorthogonal decomposition containing $\CA_i$
as a component). Indeed, spelling out~\eqref{hait}, \eqref{am}, and \eqref{ait} we obtain the following

\begin{corollary}
If $\CA \subset \D^b(X)$ is an $S$-linear admissible subcategory
such that the corresponding projection functor has finite cohomological amplitude
and $\phi:T \to S$ is a base change faithful for $f:X \to S$ then the category
\begin{equation}\label{at}
\CA_T = \{ F \in \D^b(X_T)\ |\ \phi_*(F \otimes f^*G) \in \HCA \text{ for all $G \in \D^\perf(T)$} \},
\end{equation}
{\rm(}where $\HCA$ is the minimal closed under arbitrary direct sums triangulated subcategory
of $\D_{qc}(X)$ containing~$\CA${\rm)}
is a $T$-linear admissible subcategory in $\D^b(X_T)$ such that the corresponding projection
functor has finite cohomological amplitude.
Moreover, we have $\phi^*(\CA) \subset \CA_T$ if $\phi$ has finite $\Tor$-dimension and
$\phi_*(\CA_T) \subset \CA$ if $\phi$ is projective.
\end{corollary}

\subsection{Exterior product of semiorthogonal decompositions}

Now assume that we have two algebraic varieties over the same base,
say $f:X \to S$ and $g:Y \to S$ and $S$-linear strong semiorthogonal decompositions of their derived categories
$\D^b(X) = \lan \CA_1,\dots,\CA_m \ran$ and $\D^b(Y) = \lan \CB_1,\dots,\CB_n \ran$.
Assume that their projection functors have finite cohomological amplitude.
Assume also that the cartesian square
\begin{equation}\label{xy}
\vcenter{\xymatrix{
X\times_S Y \ar[r]^p \ar[d]_q & X \ar[d]^f \\ Y \ar[r]^g & S
}}
\end{equation}
is exact, so that $g$ is a faithful base change for $f$ and $f$ is a faithful base change for $g$.
Applying Theorem~\ref{sodxt} we obtain a pair of semiorthogonal decompositions of $\D^b(X\times_S Y)$:
$$
\D^b(X\times_S Y) = \lan \CA_{1Y}, \dots ,\CA_{mY} \ran
\qquad\text{and}\qquad
\D^b(X\times_S Y) = \lan \CB_{1X}, \dots ,\CB_{nX} \ran.
$$
Let
\begin{equation}\label{aibj}
\CA_i \boxtimes_S \CB_j := \CA_{iY} \cap \CB_{jX} \subset \D^b(X\times_S Y)
\end{equation}
We call the category $\CA_i \boxtimes_S \CB_j$ {\sf the exterior product}\/ (over $S$) of $\CA_i$ and $\CB_j$.

Consider any complete order on the set $\{(i,j)\}_{1\le i\le m,\ 1\le j\le n}$
extending the natural partial order.

\begin{theorem}\label{sodxsy}
The exterior products subcategories $\CA_i \boxtimes_S \CB_j \subset \D^b(X\times_S Y)$ form a semiorthogonal decomposition
of the category $\D^b(X\times_S Y)$:
$$
\D^b(X\times_S Y) = \lan \CA_i \boxtimes_S \CB_j \ran_{1\le i\le m,\ 1\le j\le n}.
$$
Moreover, we have the following semiorthogonal decompositions
$$
\CA_{iY} = \lan \CA_i \boxtimes_S \CB_1, \dots, \CA_i \boxtimes_S \CB_n \ran
\qquad\text{and}\qquad
\CB_{jX} = \lan \CA_1 \boxtimes_S \CB_j, \dots, \CA_m \boxtimes_S \CB_j \ran.
$$
\end{theorem}
\begin{proof}
Let $\CC_{ij}^p = \lan p^*\CA_i^\perf \otimes q^*\CB_j^\perf \ran \subset \D^\perf(X\times_S Y)$
be the minimal triangulated subcategory of $\D^\perf(X\times_S Y)$ closed under taking direct summands
and containing objects of the form
$p^*A\otimes q^*B$ with $A \in \CA_i^\perf$, $B \in \CA_j^p$. The arguments of Proposition~\ref{dp}
show that $\D^\perf(X\times_S Y) = \lan \CC_{ij}^p \ran_{1\le i\le m,\ 1\le j\le n}$ is a semiorthogonal decomposition.
Moreover, it is clear from the construction that we have semiorthogonal decompositions
$$
\CA_{iY}^p = \lan \CC_{i1}^p, \dots, \CC_{in}^p \ran
\qquad\text{and}\qquad
\CB_{jX}^p = \lan \CC_{1j}^p, \dots, \CC_{mj}^p \ran.
$$
Extending these decompositions to $\D_{qc}(X\times_S Y)$ as in Proposition~\ref{sodqc}
we obtain semiorthogonal decompositions
$\D_{qc}(X\times_S Y) = \lan \HCC_{ij} \ran_{1\le i\le m,\ 1\le j\le n}$ as well as
$$
\HCA_{iY} = \lan \HCC_{i1}, \dots, \HCC_{in} \ran
\qquad\text{and}\qquad
\HCB_{jX} = \lan \HCC_{1j}, \dots, \HCC_{mj} \ran,
$$
where $\HCC_{ij}$ is obtained from $\CC_{ij}$ by addition of arbitrary direct sums and iterated addition of cones.
Finally, intersecting with $\D^b(X\times_S Y)$ we obtain semiorthogonal decompositions
$\D^b(X\times_S Y) = \lan \CC_{ij} \ran_{1\le i\le m,\ 1\le j\le n}$ as well as
$$
\CA_{iY} = \lan \CC_{i1}, \dots, \CC_{in} \ran
\qquad\text{and}\qquad
\CB_{jX} = \lan \CC_{1j}, \dots, \CC_{mj} \ran,
$$
where $\CC_{ij} = \HCC_{ij} \cap \D^b(X\times Y)$. So, it remains to check
that $\CC_{ij} = \CA_{iY} \cap \CB_{jX}$. Since $\CC_{ij} \subset \CA_{iY} \cap \CB_{jX}$
by construction it suffices to check only the other inclusion. Indeed, we have
$$
\CB_{jX} =
{}^\perp\lan \CB_{1X}, \dots, \CB_{j-1,X} \ran \cap \lan \CB_{j+1,X}, \dots, \CB_{nX} \ran^\perp =
{}^\perp\lan \CC_{it} \ran_{1\le i\le m,\ 1\le t\le j-1} \cap \lan \CC_{it} \ran^\perp_{1\le i\le m,\ j+1\le t\le n},
$$
hence
$$
\CA_{iY} \cap \CB_{jX} \subset \CA_{iY} \cap {}^\perp\lan \CC_{it} \ran_{1\le t\le j-1} \cap \lan \CC_{it} \ran^\perp_{j+1\le t\le n} = \CC_{ij}
$$
which is precisely what we need.
\end{proof}

\subsection{Products}

If $S$ is a point then any semiorthogonal decomposition of $\D^b(X)$ is $S$-linear.
Moreover, any base change $T \to S$ is flat, hence faithful for $f:X \to S$,
and $X\times_S T = X\times T$ is the product.
Thus given a semiorthogonal decomposition of $\D^b(X)$ we can construct a compatible
semiorthogonal decomposition of the bounded derived category of the product of $X$
with any quasiprojective variety. Explicitly, applying Theorem~\ref{sodxt}
we obtain the following

\begin{corollary}\label{prd}
Let $\D^b(X) = \lan \CA_1,\dots,\CA_m \ran$ be a strong semiorthogonal decomposition
the projection functors of which have finite cohomological amplitude.
Let $Y$ be a quasiprojective variety.
Then the subcategories
$$
\CA_{iY} = \{ F \in \D^b(X\times Y)\ |\ p_*(F\otimes q^*G) \in \HCA_i\ \text{for any $G \in \D^\perf(Y)$}\},
$$
where $p:X \times Y \to X$ and $q:X \times Y \to Y$ are the projections,
and $\HCA_i$ is obtained from $\CA_i$ by addition of arbitrary direct sums and iterated addition of cones,
form a $Y$-linear semiorthogonal decomposition $\D^b(X\times Y) = \lan \CA_{1Y},\dots,\CA_{mY} \ran$.
The projection functors of this semiorthogonal decomposition also have finite cohomological amplitude.
The functors $p_*:\D^b(X\times Y) \to \D_{qc}(X)$ and $p^*:\D^b(X) \to \D^b(X\times Y)$
are compatible with the semiorthogonal decompositions of $\D_{qc}(X)$ and $\D^b(X)$ respectively.
\end{corollary}

Similarly, Theorem~\ref{sodxsy} gives

\begin{corollary}\label{sodxy}
Let $\D^b(X) = \lan \CA_1,\dots,\CA_m \ran$ and $\D^b(Y) = \lan \CB_1,\dots,\CB_n \ran$
be strong semiorthogonal decompositions with projection functors of finite cohomological amplitude.
Then there is a semiorthogonal decomposition
$$
\D^b(X\times Y) = \lan \CA_i \boxtimes \CB_j \ran_{1\le i\le m,\ 1\le j\le n},
$$
where $\CA_i \boxtimes \CB_j = \CA_{iY} \cap \CB_{jX}$.
Moreover, we have semiorthogonal decompositions
$$
\CA_{iY} = \lan \CA_i \boxtimes \CB_1, \dots, \CA_i \boxtimes \CB_n \ran
\qquad\text{and}\qquad
\CB_{jX} = \lan \CA_1 \boxtimes \CB_j, \dots, \CA_m \boxtimes \CB_j \ran.
$$
\end{corollary}

\section{Correctness}

The goal of this section is to show that the extensions $\CA^\perf$, $\HCA$, $\CA^-$
of a triangulated category $\CA$ and its base change $\CA_T$ under a base change $T \to S$
(if $\CA$ is $S$-linear) do not depend on a choice of an embedding $\CA \to \D^b(X)$.
The most important technical notion for this section is that of a splitting functor.

\subsection{Splitting functors}

An exact functor $\Phi:\CT \to \CT'$ is called {\sf right splitting}\/ if
$\Ker\Phi$ is a right admissible subcategory in $\CT$,
the restriction of $\Phi$ to $(\Ker\Phi)^\perp$ is fully faithful,
and $\Im\Phi$ is right admissible in $\CT'$ (note that
$\Im\Phi = \Im(\Phi_{|(\Ker\Phi)^\perp})$
is a triangulated subcategory of $\CT'$).

\begin{lemma}[\cite{K2}]\label{sf_th}
Let $\Phi:\CT \to \CT'$ be an exact functor.
Then the following conditions are equivalent

\noindent$(1)$ $\Phi$ is right splitting;

\noindent$(2)$ $\Phi$ has a right adjoint functor $\Phi^!$
and the composition of the canonical morphism of functors
$\id_\CT \to \Phi^!\Phi$ with $\Phi$ gives an isomorphism
$\Phi \cong \Phi\Phi^!\Phi$;

\noindent$(3)$ $\Phi$ has a right adjoint functor $\Phi^!$,
there are semiorthogonal decompositions
$$
\CT = \langle\Im\Phi^!,\Ker\Phi\rangle,
\qquad
\CT' = \langle\Ker\Phi^!,\Im\Phi\rangle,
$$
and the functors $\Phi$ and $\Phi^!$ give quasiinverse equivalences
$\Im\Phi^! \cong \Im\Phi$;

\noindent$(4)$ there exists a full triangulated left admissible subcategory $\alpha:\CA \subset \CT$,
a full triangulated right admissible subcategory $\CB \subset \CT'$ and an equivalence $\xi:\CA \to \CB$
such that $\Phi = \beta\circ\xi\circ\alpha^*$, $\Phi^! = \alpha\circ\xi^{-1}\circ\beta^!$.
\end{lemma}

There is an analogous notion of left splitting functors, which enjoy a similar set of properties.
However we will not need this notion in this paper.

\subsection{Extensions}

Let $X$ be a quasiprojective variety. Let
$\alpha:\CA \to \D^b(X)$ and $\beta:\CB \to \D^b(Y)$ be admissible subcategories,
and $\xi:\CA \to \CB$ an equivalence.
Consider the corresponding right splitting functor $\Phi:\D^b(X) \to \D^b(Y)$,
$$
\Phi = \beta\circ\xi\circ\alpha^*.
$$
We assume also that $\Phi$ is {\sf geometric}, meaning that it is isomorphic to a kernel functor
$$
\Phi_\CE:\D_{qc}(X) \to \D_{qc}(Y),
\qquad
\Phi_\CE(F) = q_*(p^*F\otimes \CE)
$$
with a kernel $\CE \in \D^-(X\times Y)$.
Here $p:X\times Y \to X$ and $q:X\times Y \to Y$ are the projections.
Note that the right adjoint functor $\Phi_\CE^!$ of $\Phi_\CE$ is given by the formula
$$
\Phi_\CE^!:\D_{qc}(Y) \to \D_{qc}(X),
\qquad
\Phi_\CE^!(G) = p_*\RCHom(\CE,q^!F).
$$
It follows in particular that $\Phi_\CE$ commutes with direct sums. Indeed,
$$
\Hom(\Phi_\CE(\oplus F_i),G) \cong
\Hom(\oplus F_i,\Phi^!_\CE(G)) \cong
\prod \Hom(F_i,\Phi^!_\CE(G)) \cong
\prod \Hom(\Phi_\CE(F_i),G) \cong
\Hom(\oplus \Phi_\CE(F_i),G)
$$
implies $\Phi_\CE(\oplus F_i) \cong \oplus \Phi_\CE(F_i)$.

Recall that if $\CE \in \D^b(X\times Y)$ has finite $\Tor$-amplitude over $X$, finite $\Ext$-amplitude over~$Y$,
and $\supp\CE$ is projective over both $X$ and $Y$ then $\Phi_\CE$ takes $\D^b(X)$ to $\D^b(Y)$
and $\Phi_\CE^!$ takes $\D^b(Y)$ to $\D^b(X)$ by~\cite{K1}.

\begin{theorem}\label{extcor}
Assume that an object $\CE \in \D^b(X\times Y)$ has finite $\Tor$-amplitude over $X$, finite $\Ext$-amplitude over~$Y$,
and $\supp\CE$ is projective over both $X$ and $Y$. Assume also that the restriction of the functor
$\Phi_\CE:\D_{qc}(X) \to \D_{qc}(Y)$ to $\D^b(X)$ is a right splitting functor giving an equivalence
of subcategories $\CA \subset \D^b(X)$ and $\CB \subset \D^b(Y)$.
Then the functor $\Phi_{\CE}:\D_{qc}(X) \to \D_{qc}(Y)$ and its restriction to $\D^-(X)$
are right splitting functors giving equivalences $\HCA \cong \HCB$
and $\CA^- \cong \CB^-$.
\end{theorem}
\begin{proof}
As we already mentioned above the functor $\Phi_\CE$ commutes with direct sums.
Let us check that $\Phi_\CE^!$ also commutes with direct sums.
To do this we choose a closed embedding $i:X \to X'$ with $X'$ being smooth
and consider the functor $i_*\Phi_\CE^!$ instead. Since $i_*$ is a conservative functor
commuting with direct sums, it suffices to check that $i_*\Phi_\CE^!$ commutes with direct sums.
But it is clear that $i_*\Phi_\CE^! \cong \Phi_{(i\times\id_Y)_*\CE}^!$,
so from the whole beginning we can assume that $X$ is smooth.
Then the projection $X \times Y \to Y$ is smooth, hence
$q^!(F) \cong q^*(F)\otimes\omega_X[\dim X]$ evidently commutes with direct sums.
Further, $\CE$ is a perfect complex by~\cite{K1}, 10.46, hence
the functor $\RHom(\CE,-)$ commutes with direct sums. Finally, the functor
$p_*$ commutes with direct sums by~\cite{BV}, 3.3.4. Thus $\Phi_\CE^!$ commutes
with direct sums.

Further, since the functors $\Phi_\CE$ and $\Phi_\CE^!$ commute with direct sums
they commute with homotopy colimits by Lemma~\ref{hclc}.
%
%
Now if $F \in \D^-(X)$ then by Lemma~\ref{dmsfd} there exists a stabilizing
in finite degrees direct system of perfect complexes $F_k \in \D^b(X)$ such that $F \cong \hocolim F_k$.
Therefore $\Phi_\CE(F) \cong \Phi_\CE(\hocolim F_k) \cong \hocolim \Phi_\CE(F_k)$.
But the functor $\Phi_\CE$ has finite cohomological amplitude by Lemma~\ref{ftefca}.
Therefore the direct system $\Phi_\CE(F_k) \in \D^b(Y)$ stabilizes in finite degrees,
hence $\hocolim\Phi_\CE(F_k) \in \D^-(Y)$ by Lemma~\ref{sfd}. Thus
$\Phi_\CE$ takes $\D^-(X)$ to $\D^-(Y)$. The same argument shows that
$\Phi_\CE^!$ takes $\D^-(Y)$ to $\D^-(X)$.

To check that $\Phi_\CE$ is right splitting on $\D_{qc}(X)$ we have to check that
applying $\Phi_\CE$ to the canonical morphism of functors $\id \to \Phi_\CE^!\Phi_\CE$
gives an isomorphism $\Phi_\CE \cong \Phi_\CE\Phi_\CE^!\Phi_\CE$.
Consider the full subcategory $\CT \subset \D_{qc}(X)$ consisting of all objects $F \in \D_{qc}(X)$
for which $\Phi_{\CE}(F) \cong \Phi_{\CE}\Phi_{\CE}^!\Phi_{\CE}(F)$ in $\D_{qc}(Y)$.
We want to show that $\CT =\D_{qc}(X)$. Note that $\D^b(X) \subset \CT$ by the conditions,
and hence $\D^\perf(X) \subset \CT$. Moreover, since $\Phi$ and $\Phi^!$ commute with direct sums,
$\CT$ is closed under arbitrary direct sums. Finally, since $\Phi_\CE$ and $\Phi_\CE^!$ are exact,
$\CT$ is triangulated. So, by Lemma~\ref{perf-gen-qc} we have $\CT =\D_{qc}(X)$.


Now let us check that $\HCB = \Phi_\CE(\D_{qc}(X))$. Indeed, the RHS is contained in the LHS
by Lemma~\ref{perf-gen-qc} since $\HCB$ is closed under arbitrary direct sums triangulated
subcategory containing $\Phi_\CE(\D^\perf(X)) \subset \Phi_\CE(\D^b(X)) = \CB$.
For the other embedding it suffices to check that $\HCB$ is contained in the full
subcategory $\CT \subset \D_{qc}(Y)$ consisting of all objects $G$ such that the canonical morphism
$\Phi_\CE\Phi_\CE^!(G) \to G$ is an isomorphism. Indeed, $\CT$ contains $\CB$ by conditions of the Proposition.
Moreover, it is closed under arbitrary direct sums since both $\Phi_\CE$ and $\Phi_\CE^!$ commute with direct sums,
and is triangulated since both $\Phi_\CE$ and $\Phi_\CE^!$ are exact.
%
%
The same argument shows that $\HCA = \Im \Phi_\CE^!$, so it follows
that $\Phi_\CE$ induces an equivalence $\HCA \cong \HCB$.

Finally, since $\Phi_\CE$ and $\Phi_\CE^!$ preserve $\D^-$ and $\CA^- = \HCA \cap \D^-(X)$, $\CB^- = \HCB \cap \D^-(Y)$,
it follows that $\Phi_\CE$ induces an equivalence $\CA^- \cong \CB^-$.
\end{proof}

\begin{remark}
One can also check that $\Phi_\CE$ takes $\D^\perf(X)$ to $\D^\perf(Y)$ (this follows easily
from the fact that $\Phi_\CE^!$ commutes with direct sums). If it were also known that
$\Phi_\CE^!$ takes $\D^\perf(Y)$ to $\D^\perf(X)$ then it would follow that
$\Phi_\CE$ induces an equivalence $\CA^\perf \cong \CB^\perf$.
\end{remark}

\subsection{Base change}

Now assume that $f:X \to  S$ and $g:Y \to S$ are quasiprojective morphisms,
$\alpha:\CA \to \D^b(X)$ and $\beta:\CB \to \D^b(Y)$ are admissible $S$-linear subcategories, and
$\xi:\CA \to \CB$ is an $S$-linear equivalence.
Assume also that $\phi: T \to S$ is a base change faithful for both $f$ and $g$.
Again, consider the corresponding right splitting functor $\Phi:\D^b(X) \to \D^b(Y)$, $\Phi = \beta\circ\xi\circ\alpha^*$.
We assume also that $\Phi$ is {\sf geometrically $S$-linear}, meaning that it is isomorphic to a kernel functor
$$
\Phi_\CE:\D_{qc}(X) \to \D_{qc}(Y),
\qquad
\Phi_\CE(F) = q_*(p^*F\otimes \CE)
$$
with a kernel $\CE \in \D^-(X\times_S Y)$ supported on the fiber product of $X$ and $Y$ over $S$.
Here $p:X\times_S Y \to X$ and $q:X\times_S Y \to Y$ are the projections.
Note that the right adjoint functor $\Phi_\CE^!$ of $\Phi_\CE$ is given by the formula
$$
\Phi_\CE^!:\D_{qc}(Y) \to \D_{qc}(X),
\qquad
\Phi_\CE^!(G) = p_*\RCHom(\CE,q^!F).
$$

Consider the following commutative diagram
$$
\xymatrix{
X_T \ar[d]^\phi &
X_T\times_T Y_T \ar[d]^\phi \ar[r]^-{q_T} \ar[l]_-{p_T}
&
Y_T \ar[d]^\phi \\
X &
X\times_S Y \ar[r]^-q \ar[l]_-p &
Y
}
$$
Define the kernel $\CE_T := \phi^*\CE \in \D^-(X_T \times_T Y_T)$.

\begin{theorem}\label{bccor}
Assume that $\CE \in \D^b(X\times_S Y)$ has finite $\Tor$-amplitude over $X$, finite $\Ext$-amplitude over~$Y$,
and $\supp\CE$ is projective over both $X$ and $Y$. Assume also that $\Phi_\CE:\D^b(X) \to \D^b(Y)$
is a right splitting functor giving an equivalence of $S$-linear subcategories $\CA \subset \D^b(X)$
and $\CB \subset \D^b(Y)$. Then $\Phi_{\CE_T}:\D^b(X_T) \to \D^b(Y_T)$
is a right splitting functor inducing an equivalence $\CA_T \cong \CB_T$.
\end{theorem}
\begin{proof}
First of all note that $\CE_T$ has finite $\Tor$-amplitude over $X_T$, finite $\Ext$-amplitude over $Y$,
and projective support over both $X_T$ and $Y_T$ by~\cite{K1}, 10.47. Hence as it was mentioned in the proof
of Theorem~\ref{extcor} the functors $\Phi_\CE$, $\Phi_\CE^!$, $\Phi_{\CE_T}$, $\Phi_{\CE_T}^!$ commute
with direct sums and homotopy colimits.

Moreover, by~\cite{K1}~2.4 the functors $\Phi_\CE^!$ and $\Phi_{\CE_T}^!$ are right adjoint to $\Phi_\CE$ and $\Phi_{\CE_T}$
respectively, and all these functors preserve boundedness and coherence. Finally, by~\cite{K1}~2.42, there are canonical isomorphisms
\begin{equation}\label{fff}
\begin{array}{rclrcl}
\Phi_{\CE_T}\phi^*&=&\phi^*\Phi_{\CE},&\qquad
\Phi_{\CE}\phi_*&=&\phi_*\Phi_{\CE_T},
\\
\Phi_{\CE_T}^!\phi^*&=&\phi^*\Phi_{\CE}^!,&\qquad
\Phi_{\CE}^!\phi_*&=&\phi_*\Phi_{\CE_T}^!.
\end{array}
\end{equation}

Since $\Phi_\CE$ is right splitting on $\D_{qc}(X)$ by Theorem~\ref{extcor}, applying $\Phi_\CE$ to the canonical morphism
of functors $\id \to \Phi_\CE^!\Phi_\CE$ gives an isomorphism $\Phi_\CE \cong \Phi_\CE\Phi_\CE^!\Phi_\CE$.
Now take any $H \in \D_{qc}(X_T)$.
We want to show that $\Phi_{\CE_T}(H) \cong \Phi_{\CE_T}\Phi_{\CE_T}^!\Phi_{\CE_T}(H)$ in $\D_{qc}(Y_T)$.
By Lemma~\ref{dbdm} to do this it suffices to check that
$\phi_*(\Phi_{\CE_T}(H)\otimes g^*L^k) \cong \phi_*(\Phi_{\CE_T}\Phi_{\CE_T}^!\Phi_{\CE_T}(H)\otimes g^*L^k)$ in $\D_{qc}(Y)$
for an ample over $S$ line bundle $L$ on $T$ and any $k \gg 0$. But
\begin{multline*}
\phi_*(\Phi_{\CE_T}(H)\otimes g^*L^k) \cong
\phi_*(\Phi_{\CE_T}(H\otimes f^*L^k)) \cong
\Phi_{\CE}(\phi_*(H\otimes f^*L^k)) \cong \\ \cong
\Phi_\CE\Phi_\CE^!\Phi_{\CE}(\phi_*(H\otimes f^*L^k)) \cong
\phi_*(\Phi_{\CE_T}\Phi_{\CE_T}^!\Phi_{\CE_T}(H\otimes f^*L^k)) \cong
\phi_*(\Phi_{\CE_T}\Phi_{\CE_T}^!\Phi_{\CE_T}(H)\otimes g^*L^k).
\end{multline*}
The first and the fifth isomorphisms are given by $T$-linearity of the functors $\Phi_{\CE_T}$ and $\Phi_{\CE_T}^!$,
the second and the fourth are given by~\eqref{fff} and the third is because $\Phi_\CE$ is right splitting. So, we conclude that
$\Phi_{\CE_T} \cong \Phi_{\CE_T}\Phi_{\CE_T}^!\Phi_{\CE_T}$, hence $\Phi_{\CE_T}$ is a right splitting functor.

Now let us show that $\Phi_{\CE_T}(\D_{qc}(X_T)) = \HCB_T$.
Indeed, let $F \in \D_{qc}(X_T)$. Let $G$ be a perfect complex on~$T$. Then we have
$$
\phi_*(\Phi_{\CE_T}(F)\otimes g^*G) \cong
\phi_*(\Phi_{\CE_T}(F\otimes f^*G)) \cong
\Phi_{\CE}(\phi_*(F\otimes f^*G)) \in \HCB,
$$
hence $\Phi_{\CE_T}(F) \in \HCB_T$ by~\eqref{hait}. Further, since $\Phi_{\CE_T}$ is a right splitting $T$-linear
functor commuting with arbitrary direct sums, the category $\Phi_{\CE_T}(\D_{qc}(X_T))$ is a $T$-linear triangulated
subcategory in $\D_{qc}(Y_T)$ closed under arbitrary direct sums. On the other hand,
$$
\phi^*(\CB^\perf) \subset
\phi^*(\CB) =
\phi^*(\Phi_\CE(\D^b(X))) =
\Phi_{\CE_T}(\phi^*(\D^b(X))) \subset
\Phi_{\CE_T}(\D_{qc}(X_T))
$$
so it follows from the definition of $\HCB_T$ that $\HCB_T \subset \Phi_{\CE_T}(\D_{qc}(X_T))$.
The same argument shows that $\Phi_{\CE_T}^!(\D_{qc}(Y_T)) = \HCA_T$.

Finally, as we already mentioned the functors $\Phi_{\CE_T}$ and $\Phi_{\CE_T}^!$ preserve $\D^b$ and since
$\CA_T = \HCA_T \cap \D^b(X_T)$, $\CB_T = \HCB_T \cap \D^b(Y_T)$,
it follows that $\Phi_{\CE_T}$ induces an equivalence $\CA_T \cong \CB_T$.
\end{proof}

\section{Applications}

As an application we deduce that the projection functors of a strong semiorthogonal
decomposition are kernel functors.

\begin{theorem}
Let $X$ be a quasiprojective variety and $\D^b(X) = \lan \CA_1,\dots,\CA_m \ran$
a strong semiorthogonal decomposition. Let $\alpha_i:\D^b(X) \to \D^b(X)$ be
the projection functor to the $i$-th component. Assume that each $\alpha_i$
has finite cohomological amplitude. Then for every $i$ there is an object
$K_i \in \D^b(X\times X)$ such that $\alpha_i \cong \Phi_{K_i}$.
\end{theorem}
\begin{remark}
Note that the condition that the semiorthogonal decomposition is strong
is necessary for the projection functors to be representable by kernels.
Indeed, every functor isomorphic to $\Phi_K$ has a right adjoint functor,
hence if $\alpha_1 \cong \Phi_K$ then $\alpha_1$ has a right adjoint functor
hence $\CA_1$ is right admissible.
\end{remark}
\begin{proof}
We consider the semiorthogonal decomposition
$\D^b(X\times X) = \lan \CA_{1X},\dots,\CA_{mX} \ran$
constructed in Corollary~\ref{prd} and let $K_i$ be the component of $\Delta_*\CO_X \in \D^b(X\times X)$
in $\CA_{iX}$. Consider the corresponding filtration of $\Delta_*\CO_X$:
$$
\xymatrix@C=10pt{
0 \ar@{=}[rr] && T_m \ar[rr] && T_{m-1} \ar[rr] \ar[dl] && \dots \ar[rr] && T_1 \ar[rr] && T_0 \ar@{=}[rr] \ar[dl] && \Delta_*\CO_X \\
&&& K_m \ar@{..>}[ul] &&& \dots &&& K_1 \ar@{..>}[ul]
}
$$
Take any $F \in \D_{qc}(X)$, pull it back to $X\times X$ via the projection $p_1:X\times X \to X$,
then tensor it by the above diagram and push forward to $X$ via the projection $p_2:X\times X \to X$.
We will obtain the following diagram in $\D_{qc}(X)$
$$
\xymatrix@C=1pt{
0 \ar@{=}[rr] &&
p_{2*}(T_m\otimes p_1^*F) \ar[rr] \ar@{..>}[dr]+<-7pt,7pt>;[] &&
p_{2*}(T_{m-1}\otimes p_1^*F) \ar[rr] \ar[dl]+<7pt,7pt>  &&
\dots \ar[rr] &&
p_{2*}(T_1\otimes p_1^*F) \ar[rr] \ar@{..>}[dr]+<-7pt,7pt>;[] &&
p_{2*}(T_0\otimes p_1^*F) \ar@{=}[rr] \ar[dl]+<7pt,7pt> &&
p_{2*}(\Delta_*\CO_X\otimes p_1^*F) \\
&&& \hbox to 0pt {\hss$p_{2*}(K_m\otimes p_1^*F)$\hss} &&& \dots &&& \hbox to 0pt {\hss$p_{2*}(K_1\otimes p_1^*F)$\hss}
}
$$
Note that by Lemma~\ref{isl1} we have $K_i\otimes p_1^*F \in \HCA_{iX}$,
hence $p_{2*}(K_i\otimes p_1^*F) \in \HCA_i$ by Proposition~\ref{dqc}.
On the other hand $p_{2*}(\Delta_*\CO_X\otimes p_1^*F) \cong F$, so we conclude
that $p_{2*}(K_i\otimes p_1^*F) \cong \halpha_i(F)$.
Restricting to $\D^b(X)$ and using Lemma~\ref{cc} we obtain an isomorphism
$\Phi_{K_i} \cong \alpha_i$ on $\D^b(X)$.
%
\end{proof}

This Theorem has a relative variant.

\begin{theorem}
Let $f:X \to S$ be a morphism of quasiprojective varieties and
$\D^b(X) = \lan \CA_1,\dots,\CA_m \ran$ an $S$-linear strong semiorthogonal decomposition.
Let $\alpha_i:\D^b(X) \to \D^b(X)$ be the projection functor to the $i$-th component.
Assume that the map $f$ is faithful base change for itself and each $\alpha_i$
has finite cohomological amplitude. Then for every $i$ there is an object
$K_i \in \D^b(X\times_S X)$ such that $\alpha_i \cong \Phi_{K_i}$.
\end{theorem}

The proof is analogous. We consider the induced semiorthogonal decomposition
of $\D^b(X\times_S X)$ and consider the decomposition of $\Delta_*\CO_X$,
where this time $\Delta$ denotes the diagonal embedding into the fiber product
$\Delta:X \to X\times_S X$.

\end{document}